\documentclass[12pt,leqno,a4paper]{article}
\usepackage{color}
\usepackage{mathrsfs}
\usepackage{tikz}
\usepackage[pdfborder={0 0 0}]{hyperref}
\usepackage{framed}
\usepackage{amssymb}
\usepackage{amsmath}
\usepackage{amsthm}
\usepackage{ bbm }
\usepackage[shortlabels]{enumitem}

\newtheorem{theorem}{Theorem}
\numberwithin{theorem}{section}
\newtheorem{proposition}[theorem]{Proposition}
\newtheorem{lemma}[theorem]{Lemma}

\theoremstyle{remark}

\theoremstyle{definition}
\newtheorem{remark}[theorem]{Remark}
\newtheorem{definition}[theorem]{Definition}

\numberwithin{equation}{section}

\newcommand\set[1]{\left\{\,#1\,\right\}}		
\newcommand\abs[1]{\left|#1\right|}				
\newcommand\norm[1]{\left\Vert#1\right\Vert}	
\newcommand{\R}{\mathbb{R}}
\newcommand{\C}{\mathbb{C}}

\newcommand{\T}{\mathbb{T}}

\newcommand{\cC}{{\mathcal C}}

\newcommand{\cU}{{\mathcal U}}

\newcommand{\pa}{\partial}

\DeclareMathOperator{\sign}{sign}				
\DeclareMathOperator{\divv}{div}				

\DeclareMathOperator*{\im}{Im}
\DeclareMathOperator*{\re}{Re}

\begin{document}
\title{Entropy solutions to macroscopic IPM}
\author{\'Angel Castro \and Daniel Faraco \and Bj\"orn Gebhard }
\date{}
\maketitle

\begin{abstract}
We investigate maximal potential energy dissipation as a selection criterion for subsolutions (coarse grained solutions) in the setting of the unstable Muskat problem. We show that both,
imposing this criterion on the level of convex integration subsolutions, and the strategy of Otto based on a relaxation via minimizing movements, lead to the same nonlocal conservation law. Our main result shows that this equation admits an entropy solution for unstable initial data with an analytic interface.
\end{abstract}
%
%

\tableofcontents

\section{Introduction}\label{sec:introduction}

An outstanding open problem in hydrodynamics is the description of unstable interface configurations quickly leading to turbulent regimes. Examples are the thoroughly studied Saffman-Taylor \cite{saffman}, Rayleigh-Taylor \cite{rayleigh1883,taylor1950} and Kelvin-Helmholtz \cite{kelvinH} instabilities. In these unstable regimes Eulerian quantities, such as the velocity field, are very irregular and the Lagrangian trajectories typically  fail to be uniquely defined. Hence uniqueness is not to be expected at the microscopical level, a phenomenon that in the physics literature is known as spontaneous stochasticity \cite{physics_paper}, and instead it will be desirable to have a well-defined deterministic evolution at the macroscopical level. The current paper provides such a macroscopical evolution  in the context of the incompressible porous medium equation derived from maximal potential energy dissipation.

\subsection{IPM and interfaces}

Throughout the article, we will consider the incompressible porous media equation (IPM), given by
\begin{align}\begin{split}\label{eq:ipm}
\partial_t \rho +\divv(\rho v) &=0,\\
\divv v&=0,\\
v&=-\nabla p-\rho e_2,
\end{split}
\end{align}
on the  two-dimensional periodic strip $\T\times\R$, $\T$ denotes the flat one-torus of length $2\pi$, and over a time interval $[0,T)$, $T>0$. 
Here the (normalized) fluid density $\rho:[0,T)\times\T\times\R\rightarrow\R$, the velocity $v:[0,T)\times\T\times\R\rightarrow\R^2$ and the pressure $p:[0,T)\times\T\times\R\rightarrow\R$ are the unknowns, and $-e_2:=(0,-1)^T\in\R^2$ is the direction of gravity.

The model describes the evolution of a two-dimensional density dependent incompressible fluid in an overdamped scenario (the porous medium) and  under the influence of gravity. It consists of the law for mass conservation, the incompressibility condition for the velocity field and Darcy's law (see \cite{darcy,Muskat,saffman,allaire,Palencia,Tartar} for more physical background). Constants such as mobilities (viscosities), permeability of the medium, and  gravity have been set to $1$. System \eqref{eq:ipm} also models the motion of an incompressible and viscous fluid in a Hele-Shaw cell \cite{saffman}, a different physical scenario with the same mathematical formulation.  

Concerning initial conditions we are interested in the unstable interface case, i.e.
\begin{equation}\label{eq:initial_data}
\rho_0(x)=\begin{cases}
+1,&x_2>\gamma_0(x_1),\\
-1,&x_2<\gamma_0(x_1)
\end{cases}
\end{equation}
for a graph $\gamma_0:\T\rightarrow \R$.

Generally speaking, if the initial data $\rho_0$ is sufficiently regular it is well-known that the IPM equation has a unique regular local in time solution, see \cite{Cordoba_Gancedo_Orive,Castro_Cordoba_Gancedo_Orive}. However, the problem of formation of singularities versus global existence is still open and only partial results are known. For example, the existence of solutions with Sobolev norms unbounded in time has recently been proven in \cite{KYY}.

In the case of discontinuous initial data of the type \eqref{eq:initial_data} the situation is even more subtle as the following dichotomy shows: If the denser fluid is below the lighter one, then the problem is stable and the existence of solutions is well-known (see Section \ref{sectionmuskat}). However, if the lighter fluid is below the heavier one, the problem is ill-posed (at least in the Muskat sense, cf. Section \ref{sectionmuskat}, and in the sense of bounded weak solutions, cf. Section \ref{sec:ipm_as_differential_inclusion}). 

\subsection{Macroscopic IPM}

In spite of this difficulty, there have been several attempts to understand the evolution of such an initial configuration at least in the coarse grained picture. Namely, on one hand Felix Otto discovered that in Lagrangian formulation IPM is a gradient flow and he suggested in the unstable situation a relaxation based on the corresponding minimizing movements scheme  in the Wasserstein setting \cite{otto1999} (JKO scheme). On the other hand in \cite{Cordoba_Faraco_Gancedo} it was shown that IPM can be recast as a differential inclusion in the Tartar framework and therefore fits the De Lellis-Sz\'ekelyhidi adaptation of convex integration in hydrodynamics \cite{DeL-Sz-DI,DeL-Sz-Adm}. Subsequently, the full relaxation of the differential inclusion has been computed in \cite{Sz-ipm} leading to a concept of coarse-grained solutions (subsolutions in the convex integration jargon). In Section  \ref{sec:relaxation} we present precise definitions and review the historical landmarks of the theory. As an overview the reader can also find two diagrams, one concerning the various notions of (sub)solutions occurring in the paper and their relations, see Figure \ref{fig:table_of_solutions}, and another concerning the steps of the relaxations, see Figure \ref{fig:epic_diagram}.

Let us remark that it also has been proven in \cite{Sz-ipm} that in the case of a flat interface Otto's relaxation selects a convex integration subsolution, which turns out to be the global in time entropy solution to a one-dimensional Burgers equation, reconciling both relaxation theories. In the case of a non-flat interface, the theory of convex integration starting from \cite{Castro_Cordoba_Faraco} has provided a number of subsolutions \cite{ACF,Foerster-Sz,Noisette-Sz,Castro_Faraco_Mengual_2}. In all these situations the starting point is an ansatz for the coarse grained density $\bar \rho$ and for the growth of the mixing zone motivated in analogy to the flat case. These subsolutions show that also on a macroscopical level plenty of different evolutions are possible, such that a selection, which so far has not been available, has to be made for an attempt to claim uniqueness.

The aim of this paper is to use maximal potential energy dissipation as a selection criterion. Since as discovered by Otto in Lagrangian coordinates IPM is a gradient flow 
with respect to potential energy, this seems a natural approach. In any case, we first revisit the strategy proposed by Otto in \cite{otto1999} in the case of 
non flat interfaces (the scheme is explained in Section \ref{sec:otto} and Appendix \ref{sec:otto_relaxation_appendix}).
We then  reconcile it by selecting the subsolution in the convex integration terminology which at each time instant dissipates the most potential energy. 

It can be shown, that both, the relaxed minimizing movements scheme provided by  \cite{otto1999} (at least formally) and imposing maximal potential energy dissipation among convex integration subsolutions (rigorously), lead to the following equation 
\begin{align}\begin{split}\label{eq:subequation_inintroduction}
    \partial_t\rho+\divv\left(\rho v\right)+\partial_{x_2}(\rho^2)&=0,\\
    \divv v&=0,\\
    v&=-\nabla p-\rho e_2,
    \end{split}
\end{align}
which will be referred to as  {\it macroscopic IPM}. In  Section \ref{sec:relaxation} we explain in detail how the Muskat problem, the theory of subsolutions, convex integration for IPM and Otto's relaxation are connected. The derivation of \eqref{eq:subequation_inintroduction} from the 
JKO scheme is known to experts, but as far as we are aware the arguments around maximal potential energy dissipation for subsolutions are new.
In particular, it will be explained in which way \eqref{eq:subequation_inintroduction} can offer a  selection criterion for IPM subsolutions based on a natural extension of the gradient flow structure of IPM.

(Entropy) solutions to macroscopic IPM are subsolutions to IPM as long as they exist. By introducing a parameter $0<\mu<1$ in the first equation,
\begin{align*}
    \partial_t\rho+\divv(\rho v+\mu\rho^2e_2)=0
\end{align*}
macroscopic IPM produces strict subsolutions. Hence by 
a suitable $h$-principle, cf. Theorem \ref{thm:CI}, the time of existence of microscopical solutions to IPM
will be dictated by the time of existence of \eqref{eq:subequation_inintroduction}. This is in stark contrast to \cite{Castro_Cordoba_Faraco} where a rarefaction like ansatz with a prescribed speed of opening of the mixing zone is made and a resulting time and space dependent parameter $\mu(t,x)$ is derived which is smaller than $1$ just for short times.

In general we emphasize that contrary to the procedure of \cite{Castro_Cordoba_Faraco} (and also of \cite{ACF,Foerster-Sz,Noisette-Sz,Castro_Faraco_Mengual_2}), i.e. deriving a macroscopic equation from an ansatz, we here follow the reversed process, i.e. we consider based on a selection a fixed equation for the macroscopic evolution and derive properties of its solutions, such as the speed of opening of the mixing zone. We believe that this is a necessity when it comes to potential applications addressing for instance the prediction of a unique mixing zone evolution.

\subsection{Existence result and idea of proof}

The bulk of the paper is devoted to proving the existence of an entropy solution for \eqref{eq:subequation_inintroduction} with \eqref{eq:initial_data} as initial data. System \eqref{eq:subequation_inintroduction} can be written as a single  scalar nonlocal hyperbolic conservation law,
\begin{align}\label{eq:nonlocal_hyperbolic_law}
    \partial_t\rho+\divv(\rho T[\rho])+\partial_{x_2}(\rho^2)=0,
\end{align}
where $v=T[\rho]$ is a $0$-order singular integral operator. Contrary to other nonlocal conservation laws with a more regular nonlocal feedback, see \cite{amadori2012,amorim2012,betancourt2011,blandin2016,colomboR2012} for examples and \cite{nonlocal_overview} for a recent overview, a general existence and uniqueness theory for nonlocal terms as in \eqref{eq:nonlocal_hyperbolic_law} is not available. We bypass this by using the structure of the two-phase initial data \eqref{eq:initial_data}. This approach, born out of necessity, not only provides us with the existence of a solution, but in addition allows us to learn about certain properties of it. More precisely, by showing that the Burgers' term $\partial_{x_2}(\rho^2)$  is able to  tear up the initial discontinuity of the density even in the presence of the incompressible velocity $v$, we will prove the existence of a local in time solution which is Lipschitz for $t>0$. This fact is highly nontrivial and presents many technical difficulties that will be tackled in Sections \ref{sec:level_set_formulation} to \ref{sec:justification}, which together form the proof of our main theorem and will be described below. A careful statement of our main theorem itself, containing further properties of the solution, can be found in Section \ref{sec:main_result}.We have preferred to state the existence theorem for \eqref{eq:subequation_inintroduction} after the reader is hopefully convinced by Section \ref{sec:relaxation} that \eqref{eq:subequation_inintroduction} renders a macroscopic description for the unstable Muskat problem consistent with maximal potential energy dissipation.

One main ingredient of our proof is to look at the evolution of level sets of the density $\rho$ in suitably scaled coordinates and to adjust properly to leading order terms of this evolution. These steps carried out in Section \ref{sec:level_set_formulation}, reduce the initial value problem \eqref{eq:initial_data}, \eqref{eq:subequation_inintroduction} to a fixed point problem of the type
\begin{align}\label{eq:fixed_point_equation_inintroduction}
    \eta(t,y)=\frac{1}{t^{1+\alpha}}\int_0^t\int_{-2}^2\int_{\T}\big(K_s[\eta(s,\cdot)]\big)(y,z)\big(h_s[\partial_{y_1}\eta(s,\cdot)]\big)(y,z)\:dz_1\:dz_2\:ds- \frac{1}{t^{\alpha}}h_0(y),
\end{align}
for functions $\eta:[0,T)\times \T\times(-2,2)\rightarrow \R$ describing the evolution of the level sets in superlinear order with respect to $t>0$ small. The constants $\pm 2$ for the domain of $\eta$ are coming from the rarefaction speed of Burgers' equation. Moreover, here $h_0(y)$ is one of the mentioned leading order terms, in fact the first order term, depending on the initial graph $\gamma_0$, and $\alpha\in(0,1)$. Moreover, for each $s>0$, $y\in\T\times(-2,2)$ and $\xi:\T\times(-2,2)\rightarrow \R$ fixed, the function $z\mapsto\big(K_s[\xi]\big)(y,z)$ is an order $-1$ convolution kernel induced by the Biot-Savart law. The dependence on $\xi$ involves both $\xi(z)$ and $\xi(y)$ in the form of the difference $\xi(y)-\xi(z)$. Similar the function $(y,z)\mapsto h_s[\partial_{y_1}\xi](y,z)$, again considered for a fixed $s$ and $\xi$, depends on the difference $\partial_{y_1}\xi(y)-\partial_{y_1}\xi(z)$. Thus after integration in $z$ the regularity of the right-hand side of \eqref{eq:fixed_point_equation_inintroduction} with respect to $y$ is the regularity of $\partial_{y_1}\eta$, i.e. the right-hand side when seen as an operator looses one derivative in $y_1$.

In addition, as one of the main difficulties, also for potential equivalent reformulations of \eqref{eq:fixed_point_equation_inintroduction} where the above loss of a derivative might be avoided, we like to point out that the kernels $K_s[\xi]$ degenerate as $s\rightarrow 0$ to a one-dimensional kernel with singularity $\sim 1/(y_1-z_1)$, i.e. to an order $0$ integral kernel. Thus estimates for $K_s[\xi](y,\cdot)$ as an order $-1$ kernel, and thus after integration compensating the loss $\partial_{y_1}\xi(z)$ (but not $\partial_{y_1}\xi(y)$ though), cannot be obtained uniformly in $s$.

Anyhow, in order to keep the paper enjoyable, we deal with \eqref{eq:fixed_point_equation_inintroduction} and its loss of derivative by considering real analytic initial interfaces. This allows us to use an adaptation of the Nirenberg-Nishida abstract Cauchy-Kovalevskaya Theorem. Still, the application of it, even when we continue to ignore the so far not mentioned factor $t^{-(1+\alpha)}$ on the right-hand side, takes quite a lot of effort.  It is the second main part of our proof and can be found in Section \ref{sec:solving}.

Finally, Section \ref{sec:justification} puts everything together to give a solution to the macroscopic IPM equation. 
In the appendices \ref{sec:CK_proof},\ref{sec:otto_relaxation_appendix},\ref{sec:rigorous_dissipation} we give a proof of a version of the abstract Cauchy-Kovalevskaya Theorem needed for our situation and we give some more details regarding the derivation of the macroscopic IPM equation.

\subsection{On the entropy condition}
We emphasize that the solution we find is an entropy solution of \eqref{eq:subequation_inintroduction}, or rather \eqref{eq:nonlocal_hyperbolic_law}. The notion of entropy solution is stated in Definition \ref{def:entropy_solution}. This is consistent with the flat case $\gamma_0=0$ where, as said earlier, the relaxed minimizing movements scheme of Otto \cite{otto1999} converges to the entropy solution of \eqref{eq:nonlocal_hyperbolic_law} which in that case reduces to Burgers' equation. For an extended discussion concerning the selection of the entropy solution by the minimizing movements scheme (including other gradient flows as counterexamples where a corresponding selection fails) we refer to \cite{GigliOtto} where Gigli and Otto have revisited the IPM relaxation in the flat setting of Otto's original work \cite{otto1999}. In addition see also \cite{otto2001} for a stability result in the flat case. Concerning the general, non-flat case it was also conjectured by Otto \cite{otto_conjecture} that the convergence of the minimizing movements scheme to an entropy solution remains true.

\begin{figure}
\begin{framed}
\noindent\textbf{(non-strict) IPM subsolutions}\\
Definition \ref{def:subsolutions} (non-strict: with ``$\leq$'' instead of ``$<$'' in \eqref{eq:pointwise_constraint_relaxed})\\
existence: \cite{ACF,Castro_Cordoba_Faraco,Foerster-Sz,Noisette-Sz} and boxes below\\
uniqueness: no
\begin{framed}
    \noindent\textbf{macroscopic IPM solutions}\\
    Definition \ref{def:solution_ipmsub}\\
    existence: boxes below\\
    uniqueness: no
\begin{framed}
    \noindent\textbf{IPM solutions}\\
    Definition \ref{def:solution_ipm}\\
    existence: \cite{Siegel_ill} and induced by subsolutions via \cite{Sz-ipm}\\
    uniqueness: no
\end{framed}
\begin{framed}
    \noindent\textbf{macroscopic IPM entropy solutions}\\
    Definition \ref{def:entropy_solution}\\
    existence: Theorem \ref{thm:main}\\
    uniqueness: open
\end{framed}
\end{framed}
\end{framed}
\caption{Relation of (sub)solutions in the unstable non-flat interface case: Note that each IPM solution is indeed also a macroscopic IPM solution due to the fact that IPM solutions satisfy $\rho(t,x)^2=1$ for almost every $(t,x)$. Concerning the strictness of the stated inclusions, the listed references \cite{ACF,Castro_Cordoba_Faraco,Foerster-Sz,Noisette-Sz} provide subsolutions different from macroscopic IPM solutions. We believe that also the other two inclusions are strict, e.g. by methods similar to the ones used in the present paper it should be possible to construct a non-flat two-shock solution to macroscopic IPM that is neither an entropy solution nor an IPM solution.}
\label{fig:table_of_solutions}
\end{figure}

Moreover, we point out that some sort of choice among solutions of \eqref{eq:subequation_inintroduction} is critical in order to have a selection criterion. Indeed, already in the flat case solutions are clearly not unique, and also in the general case nonentropic solutions for  \eqref{eq:subequation_inintroduction} can be obtained in an easier way, for instance via \eqref{eq:muskat_equation} below, cf. Remark \ref{rem:remark_after_mainthm}. We believe that the requirement of being an entropy solution leads to uniqueness for the intial value problem \eqref{eq:initial_data}, \eqref{eq:subequation_inintroduction}, but  since the velocity $v$ depends on $\rho$ in a comparably singular nonlocal way, standard methods do not seem to work and uniqueness of entropy solutions to macroscopic IPM stands as an interesting open question. In any case, we emphasize that for the scheme we present there is a unique solution, and therefore our maximal dissipating subsolution is amenable to numerical calculations.

\subsection{Further questions}

Besides the question of uniqueness of the found entropy solution, our work opens the door to many other questions with various levels of difficulty, such as improving the regularity of the solutions, considering initial interfaces (not being analytical or not being a graph, as for example in \cite{Castro_Faraco_Mengual_2}) or other densities as initial data as well. It would be interesting to see whether the JKO scheme does converge rigorously or what happens in the case of different mobilities \cite{Mengual,otto1999}. On a more general level, there might be other selection criteria for IPM, for example based on surface tension \cite{surface_tension} or on vanishing diffusion \cite{menon_otto_2005,MenonOtto}. Finally, we emphasize that our selection criterion ultimately is tailored to the gradient flow structure of IPM, and for other equations the reasoning necessarily must be different. In any case, we hope our work encourages the research on finding a deterministic coarse grained evolution in presence of instabilities.

\section{Ill-posedness and relaxation}\label{sec:relaxation}

 The here considered unstable interface initial value problem is highly ill-posed. In this section we explain in which sense this ill-posedness holds, as well as a strategy based on convex integration and the relaxation of Otto \cite{otto1999} to overcome it.
 This section, having the purpose to fully motivate equation \eqref{eq:subequation_inintroduction}, is mostly a review of existing results. 
 Except for the derivation in Section \ref{sec:transfer_to_subsolutions} showing that maximal potential energy dissipating subsolutions coincide with Otto's relaxation, we do not claim any novelty. However, we are not aware that the computations in Section \ref{sec:otto} can be found in the literature.
 A reader only interested in solving system \eqref{eq:subequation_inintroduction} can go directly to Section \ref{sec:main_result}.

\subsection{The Muskat problem}\label{sectionmuskat}

If one assumes that 
\begin{align*}
    \rho(x,t)=\left\{\begin{array}{cc} \rho_{\text{up}} &  x_2>f(x_1,t)\\
    \rho_{\text{down}} & x_2<f(x_1,t)\end{array}\right.,
\end{align*}
 a closed equation from \eqref{eq:ipm} can be obtained for the interface $(x_1, f(x_1,t))$. Indeed,
\begin{align}\label{eq:muskat_equation}
    \pa_t f(x,t)=\frac{\rho_{\text{down}}-\rho_{\text{up}}}{4\pi}\int_{\mathbb{T}}
    \frac{\sin(y)(f_x(x,t)-f_x(x-y,t))}{\cosh(f(x,t)-f(x-y,t))-\cos(y)}dy.
\end{align}
This equation is usually known in the literature as the Muskat equation honoring M. Muskat \cite{Muskat}. 

In the  case $\rho_{\text{down}}>\rho_{\text{up}}$ the problem is stable and local existence and regularity of solutions can be proven in different functional settings and situations \cite{Cordoba_Gancedo,Cordoba_Cordoba_Gancedo1,Cordoba_Cordoba_Gancedo2,GS,Cordoba_Granero_Orive,fanghua,alazardlazar,nguyenpausader,AlazardNguyenpp1,AlazardNguyenpp2,AlazardNguyenpp3,Alazard_Nguyen4,cameron,EM,matioc,Abels_Matioc,Chen_Nguyen_Xu,Shi,muskat-self-similar,garciajuarez2023desingularization,Choi_Jerison_Kim,Agrawal_Patel}, as well as global for small and medium size initial data \cite{CCGPS,CCGS,CGSV,gancedolazar,CordobaLazar20,Granero_Lazar,Alonso-Oran_Granero,dong2021global}. The existence of singularities for large initial data is shown in \cite{CCFGM,CCFG} and also in \cite{Cordoba_Serrano_Zlatos1,Cordoba_Serrano_Zlatos2}.

However, if $\rho_{\text{down}}<\rho_{\text{up}}$ the Muskat equation is ill-posed \cite{Siegel_ill,Cordoba_Gancedo}. Surprisingly, convex integration has allowed us to construct solutions to IPM starting in this kind of unstable situations. They have been called mixing solutions and, in them, the initial interface between the two different densities disappears and a strip arises in which the two densities mix. We elaborate on these mixing solutions in the next sections.
For a general picture on convex integration in the context of fluid dynamics we refer to the surveys \cite{Buckmaster-Vicol-Survey,DeL-Sz-Survey2,DeL-Sz-Survey}.

\subsection{IPM as differential inclusion}\label{sec:ipm_as_differential_inclusion}

The first examples of nonuniqueness of weak solutions for \eqref{eq:ipm} using convex integration were given in \cite{Cordoba_Faraco_Gancedo} by C\'ordoba, Gancedo and the second author for the initial value $\rho_0=0$. Their method bypasses the computation of the relaxation by means of so-called $T_4$ configurations. After this Sz\'ekelyhidi established in \cite{Sz-ipm} the explicit relaxation of \eqref{eq:ipm} for initial data of two-phase type, enabling a systematic investigation of interface problems in IPM. 
While the results in \cite{Cordoba_Faraco_Gancedo,Sz-ipm} established ill-posedness of IPM in the class of essentially bounded solutions, Isett and Vicol could also show the existence of compactly supported $\cC_{t,x}^\alpha$-solutions for $\alpha<1/9$, \cite{Isett_Vicol}. The starting point of our investigation is the relaxation of \cite{Sz-ipm}, which we will describe in this subsection.

In the following we consider initial data with $\abs{\rho_0}=1$ almost everywhere. The corresponding notion of weak solutions is fixed in Definition \ref{def:solution_ipm} below. Note that for such initial data the last condition in the definition is an additional consistency requirement coming from the continuity, or rather transport, equation in \eqref{eq:ipm}.
\begin{definition}\label{def:solution_ipm}
A pair $\rho\in L^\infty((0,T)\times\T\times\R)$, $v\in L^\infty(0,T;L^2(\T\times\R;\R^2))$ is a solution of \eqref{eq:ipm}, \eqref{eq:initial_data} provided for any $\varphi\in\cC_c^\infty([0,T)\times\T\times\R)$ there holds
\begin{align*}
    \int_0^T\int_{\T\times\R}\rho\partial_t\varphi+\rho v\cdot\nabla \varphi\:dx\:dt+\int_{\T\times\R}\rho_0\varphi(0,\cdot)\:dx=0,\\
    \int_0^T\int_{\T\times\R}v\cdot \nabla\varphi\:dx\:dt=0,\\
    \int_0^T\int_{\T\times\R}(v+\rho e_2)\cdot\nabla^\perp \varphi\:dx\:dt=0,
\end{align*}
and $\abs{\rho(t,x)}=1$ for almost every $(t,x)\in (0,T)\times\T\times\R$.
\end{definition}

A key step in the works \cite{Cordoba_Faraco_Gancedo,Sz-ipm} is to recast weak solutions as defined above as solutions to a differential inclusion, to be able to use the Murat-Tartar compensated compactness formalism \cite{Tartar1979}.

A pair $(\rho,v)$ is a weak solution if and only if the triple $(\rho,v,m)\in L^\infty((0,T)\times\T\times\R)\times (L^\infty(0,T;L^2(\T\times\R)))^2$ satisfies the linear system 
\begin{align}\begin{split}\label{eq:linear_system}
    \partial_t\rho+\divv m&=0,\\
    \divv v&=0,\\
    v&=-\nabla p-\rho e_2,\\
    \rho(0,\cdot)&=\rho_0
    \end{split}
\end{align}
distributionally, i.e. in analogy to Definition \ref{def:solution_ipm}, together with 
\begin{align}\label{eq:pointwise_constraint}
    \big(\rho(t,x),v(t,x),m(t,x)\big)\in K:=\set{(\rho,v,m)\in\R^5:\abs{\rho}=1,~m=\rho v}
\end{align}
for almost every $(t,x)\in(0,T)\times\T\times\R$. 

Then the relaxation of the incompressible porous media equation is understood as the relaxation of the corresponding differential inclusion, i.e. in the pointwise nonlinear constraint \eqref{eq:pointwise_constraint} the set $K$ is replaced by its convex (or more generally $\Lambda$-convex) hull. Up to technicalities one can recover highly oscillatory solutions from this set, as the main theorem of Sz\'ekelyhidi in \cite{Sz-ipm} shows.

\begin{theorem}[\cite{Sz-ipm}]\label{thm:CI} Let $\bar \rho\in L^\infty((0,T)\times\T\times\R)$ and $\bar v,\bar m\in L^\infty(0,T;L^2(\T\times\R))$ satisfy \eqref{eq:linear_system} in the sense of distributions. Suppose that there exists a bounded and open set $\mathscr{U}\subset (0,T)\times\T\times\R$ such that there holds \eqref{eq:pointwise_constraint} for almost every $(t,x)\notin\mathscr{U}$, while $(\bar{\rho},\bar{v},\bar{m})$ are continuous on $\mathscr{U}$ with 
\begin{align}\begin{split}\label{eq:pointwise_constraint_relaxed}
    \big(\bar\rho&(t,x),\bar v(t,x),\bar m(t,x)\big)\\
    &\in \set{(\rho,v,m)\in\R^5:\abs{\rho}<1,~\abs{2(m-\rho v)+(1-\rho^2)e_2}<(1-\rho^2)}\end{split}
\end{align}
for every $(t,x)\in\mathscr{U}$. Then there exist infinitely many weak solutions $(\rho,v)$ of \eqref{eq:ipm}, \eqref{eq:initial_data} that coincide with $(\bar{\rho},\bar{v})$ outside of $\mathscr{U}$ and are arbitrarily close to $(\bar{\rho},\bar{v})$ in the weak $L^2(\mathscr{U})$-topology.
\end{theorem}

In the case of the IPM system the set on the right-hand side of \eqref{eq:pointwise_constraint_relaxed} is indeed only the interior of the $\Lambda$-convex hull of $K$, see \cite{Sz-ipm} for a precise definition, which does not coincide with the full convex hull as opposed to the Euler equations. Still, \eqref{eq:pointwise_constraint_relaxed} describes all possible weak limits of solutions to the IPM system, cf. \cite{Sz-ipm}. In view of that one can therefore truly speak about the full relaxation of IPM in the context of two-phase mixtures.

This fact has been quantified in \cite{Castro_Faraco_Mengual}, where the relation 
between solutions and subsolutions has been made precise through an adapted $h$-principle. In particular, this leads to additional properties of the solutions like a degraded macroscopic behaviour or the turbulent mixing at every time-slice property. The latter means that the solutions $(\rho,v)$ induced by $(\bar{\rho},\bar{v},\bar{m})$ satisfy $\rho\in\cC^0([0,T);L^2_{weak}(\T\times(-R,R)))$, where $R$ is some positive number with $\mathscr{U}\subset(0,T)\times\T\times(-R,R)$, and 
\begin{equation}\label{eq:turbulent_mixing_at_every_time_slice}
    \left(\int_{B}(1-\rho(t,x))\:dx\right) \left(\int_{B}(1+\rho(t,x))\:dx\right)>0
\end{equation}
for any $t\in(0,T)$ and any ball $B$ fully contained in $\mathscr{U}_t:=\set{x\in\T\times\R:(t,x)\in\mathscr{U}}$.

For later purpose we also point out the following possible upgrade of Theorem \ref{thm:CI} which is obtained by using convex integration as in \cite{Castro_Faraco_Mengual,DeL-Sz-Adm}.
\begin{lemma}\label{lem:CI_regarding_potential_energy}
Let $(\bar{\rho},\bar{v},\bar{m})$ be as in Theorem \ref{thm:CI} and $\delta:[0,T)\rightarrow \R$ continuous with $\delta(0)=0$, $\delta(t)>0$, $t>0$. Then there exist infinitely many solutions $(\rho,v)$ as in Theorem \ref{thm:CI} with the additional property that 
\begin{align*}
    \abs{\int_{\T\times\R}(\bar{\rho}(t,x)-\rho(t,x))x_2\:dx}\leq \delta(t)
\end{align*}
for almost every $t\in[0,T)$.
\end{lemma}

\begin{definition}\label{def:subsolutions}
Any triple $(\bar{\rho},\bar{v},\bar{m})$ satisfying the conditions of Theorem \ref{thm:CI} is called a subsolution of \eqref{eq:ipm}, \eqref{eq:initial_data}. The set $\mathscr{U}$, in other papers frequently also denoted by $\Omega_{\text{mix}}$, is called the mixing zone of the subsolution.
\end{definition}

Theorem \ref{thm:CI} shifts the focus from a single solution to the investigation of subsolutions which are understood as possible coarse grained or averaged solutions. As subsolutions play the central role also in the present investigation we will frequently omit the bars in notation and instead mark solutions by $(\rho_{sol},v_{sol})$ in case there is a chance of confusion.

\subsection{Examples of subsolutions}\label{sec:examples_of_subsolutions}
The first examples of non constant subsolutions have been given in the same paper \cite{Sz-ipm} of Sz\'ekelyhidi for the perfectly flat initial interface, $\rho_0(x)=\sign(x_2)$. Keeping the one-dimensional structure of the initial data one sees that $v=0$, $m=-\alpha(1-\rho^2)e_2$, $\alpha\in(0,1)$ reduces \eqref{eq:linear_system}, \eqref{eq:pointwise_constraint} to the one-dimensional conservation law
\begin{align*}
    \partial_t\rho +\alpha\partial_{x_2}(\rho^2)=0,
\end{align*}
which has a unique entropy solution given by
\begin{align*}
    \rho(t,x)=\begin{cases}
    1,&x_2>2\alpha t,\\
    \frac{x_2}{2\alpha t},&-2\alpha t<x_2<2\alpha t,\\
    -1,&x_2<-2\alpha t.
    \end{cases}
\end{align*}
It also has been mentioned in \cite{Sz-ipm} that the limiting case $\alpha=1$ is in agreement with the relaxation of Otto \cite{otto1999}. It coincides with \eqref{eq:subequation_inintroduction} in the flat situation, cf. Section \ref{sec:flat_interface}. In addition this case gives an upper bound for the the mixing zone. More precisely, it has been shown in \cite{Sz-ipm} that the mixing zone at time $t>0$, $\mathscr{U}_t$ of any one-dimensional subsolution emanating from $\rho_0(x)=\sign(x_2)$ is contained in the strip $[-1,1]\times(-2t,2t)$. A similar subsolution in the harder case of different viscosities was studied in \cite{Mengual}. Actually in \cite{Mengual} the $\Lambda$-hull of IPM with different viscosities and densities is computed. 

In the context of IPM and differential inclusions we also like to mention the article \cite{Hitruhin-Lindberg} which addresses the stationary, i.e. time-independent, IPM system. In \cite{Hitruhin-Lindberg} the lamination convex hull of that system is computed, as well as a rigidity result for its subsolutions and an application for long-term limits of \eqref{eq:ipm} given.

The first examples of subsolutions giving rise to mixing solutions, i.e. solutions obtained from the subsolution via convex integration with property \eqref{eq:turbulent_mixing_at_every_time_slice}, for IPM starting in a non flat interface $(x_1, f_0(x_1))$ have been provided in \cite{Castro_Cordoba_Faraco}. In this paper the density $\rho$ of the subsolution is Lipschitz and the prescribed speed of opening of the mixing zone $c(x_1)$ ($=2\alpha$ in the flat case above), satisfies $1\leq c <2$ and, as indicated, might depend on $x_1$.  The result of \cite{Castro_Cordoba_Faraco} holds for initial data $f_0\in H^5(\mathbb{R})$, i.e. in a regime where the Muskat problem can not be solved. A numerical analysis of these subsolutions can be found in \cite{C} where the formation of fingers can be observed.  In \cite{ACF}, 
the semiclassical viewpoint developed in \cite{Castro_Cordoba_Faraco}
is taken one step further (using semiclassical Sobolev spaces for example), providing an alternative proof
to the main result of \cite{Castro_Cordoba_Faraco}. Indeed this later approach improves the subsolutions with respect to their regularity, as the boundary of the mixing zone is in $H^{5-\frac{1}{c(x_1)}}$, where $c(x_1)$ is the local speed of opening of the mixing zone, instead of merely in $H^4$

In \cite{Foerster-Sz} the authors constructed mixing solutions with an initial interface $f_0\in \cC^{3+\alpha}$ relaxing  the initial regularity needed in   \cite{Castro_Cordoba_Faraco} but relying on subsolutions with piecewise-constant density instead of Lipschitz. In this case the speed of opening of the mixing zone is $0<c<2$ with $c$ uniform in $x_1$. Thereafter the same kind of subsolutions have been constructed in \cite{Noisette-Sz} with variable speed of opening. 

As mentioned before, mixing solutions obtained via convex integration are not unique. There are two reasons for this fact: a) different subsolutions can be found, b) infinitely many solutions, corresponding to different distributions of the density, emanate from every fixed subsolution. In order to deal with the point b), in \cite{Castro_Faraco_Mengual} it has been shown that all the solutions obtained from a fixed subsolution can be chosen in such a way that they share averages over large sets, i.e., they are the same as the subsolution at a macroscopic level. One of the main points of the present paper is to deal with point a). A particular instance of this multiplicity will be illustrated in Subsection \ref{sec:subsolution_selection_problem} below.

The constructions of the subsolutions above seem to rely on the Saffman-Taylor instability
(heavy fluid on top of a lighter fluid). In \cite{Castro_Cordoba_Faraco} it was observed
that there also exist mixing solutions in the stable regime (see also \cite{Foerster-Sz})  which build on Kelvin-Helmholtz type instabilites (discontinuity of the velocity field, see \cite{Mengual} for a thorough discussion of this phenomena at the level of the hulls). Actually, the analysis in \cite{Castro_Cordoba_Faraco} indicates that the  mixing can be created around  any point of the interface which is not both flat (with zero slope) and stable. We call points having zero slope in the stable regime fully stable points. 
It happens that in an initially overhanging interface there must be always a fully stable point. Partially unstable situations therefore require to find compatibility between the Muskat solution and mixing solutions, see \cite{Castro_Faraco_Mengual_2}. 
Remarkably, the construction in  \cite{Castro_Faraco_Mengual_2}  allows to answer the question on how to prolongate in time the singular solutions to the Muskat problem  found in \cite{CCFGM,CCFG}, namely as mixing solutions. 

As a last remark we like to point out that the subsolutions constructed in \cite{Castro_Cordoba_Faraco,Foerster-Sz,Noisette-Sz,Castro_Faraco_Mengual_2} are local in time in the sense that although the involved functions exist over a potentially larger time interval, a small time interval has to be chosen in order to guarantee that they take values inside the convex hull, i.e. that \eqref{eq:pointwise_constraint_relaxed} holds true. This is in contrast to the flat cases \cite{Sz-ipm,Mengual} and to the subsolution constructed in the present paper. Although also here we will only prove a local in time existence result, the involved functions take values in (the closure) of the convex hull as long as they exist. 

\subsection{The subsolution selection problem}\label{sec:subsolution_selection_problem}
As described, the constructions from the previous subsection contain ansatzes for certain properties of the subsolution and hence for the induced mixing solutions of \eqref{eq:ipm}. To illustrate this freedom in the simplest case let us discuss the flat interface with $\gamma_0(x_1)=0$ in slightly more detail. As in \cite{Sz-ipm}, setting $v\equiv 0$,  $m=m_2(t,x_2)e_2$, $\rho=\rho(t,x_2)$ one sees that $(\rho,0,m)$ is a subsolution if and only if 
\begin{gather*}
    \partial_t\rho+\partial_{x_2}m_2=0,\quad \rho(0,x)=\sign(x_2),
    \\    \abs{\rho}\leq 1, \\
    \abs{2m_2+1-\rho^2}<1-\rho^2, \text{ when }\abs{\rho}<1,\\
    m_2=0, \text{ when }\abs{\rho}=1
\end{gather*}
and the required continuity conditions hold true. Thus one could make the ansatz
\begin{align}\label{eq:ansatz_for_m_in_flat_case}
    m_2=-\frac{1-\rho^2}{2}+\frac{1-\rho^2}{2}\xi_2
\end{align}
with $\xi_2:[0,T)\times\R\rightarrow\R$ satisfying $\abs{\xi_2}<1$ and for any such $\xi_2$ solve the conservation law
\begin{align*}
    \partial_t\rho+\partial_{x_2}\left((\xi_2(t,x_2)-1)\frac{1-\rho^2}{2}\right)=0
\end{align*}
with initial data $\rho_0(x_2)=\sign(x_2)$ to get plenty of subsolutions with different mixing zones and density profiles. Note that in that sense $\xi_2$, or rather the whole relation \eqref{eq:ansatz_for_m_in_flat_case}, plays the role of a constitutive law.

Summarizing once more, these examples show that not only each subsolution induces infinitely many solutions of the incompressible porous media equation sharing a common coarse grained, or averaged, behaviour, but that there are also infinitely many possibilities for this averaged evolution via the vast amount of possible subsolutions. This is a common problem in the construction of turbulent solutions emanating from unstable interface initial data, as for instance also for the Kelvin-Helmholtz instability \cite{GK_energy_euler,Mengual_Sz_vortex_sheet,Sz-KH} and the Rayleigh-Taylor instability in the context of the Euler equations \cite{GKBou,GKHirsch,GKSz}. We emphasize that however our criteria builds on the gradient flow structure of IPM, and therefore different ideas should be used in the case of the Euler equations, see Section \ref{sec:comparison} for a short overview of so far used strategies.

\subsection{A selection criterion}\label{sec:selection}

We now focus in the general, not necessarily flat, case on the selection of subsolutions in terms of choosing an appropriate relation between $m$, $\rho$ and $v$ such that \eqref{eq:pointwise_constraint_relaxed} holds true provided $\abs{\rho}\leq 1$. 

First we will review  the strategy proposed by F. Otto in \cite{otto1999} to relax system \eqref{eq:ipm} based on its gradient flow structure in Lagrangian coordinates and we will formally obtain \eqref{eq:subequation_inintroduction} from this relaxation. The strategy of Otto does not rely on the notion of a subsolution in the context of differential inclusions as in Section \ref{sec:relaxation}. However, the solution of \eqref{eq:subequation_inintroduction} will be a (non-strict) subsolution with $m=\rho v-(1-\rho^2)e_2$.

Thereafter, we will also give an argument to derive \eqref{eq:subequation_inintroduction} in Eulerian coordinates directly based on subsolutions. Also here the starting point will be the gradient flow structure of \eqref{eq:ipm}. This second argument shows that the relaxation of Otto selects among all subsolutions precisely those that maximize the dissipation of potential energy at every time instant.

The relations are summarized in Figure \ref{fig:epic_diagram}.

\subsubsection{Otto's relaxation}\label{sec:otto}

In this section we give a very brief summary of Otto's 5 step strategy leading to the macroscopic IPM equation \eqref{eq:subequation_inintroduction}. The discussion is not rigorous and even then we have put most of the explicit calculations to Appendix \ref{sec:otto_relaxation_appendix}. We adapt our notation to the one of \cite{otto1999}, which, due to a different normalization, studies the evolution of 
$$
s(x,t)=\frac{1-\rho(x,t/2)}{2}
$$
instead of $\rho(x,t)$, i.e. contrary to other sections the density $s$ is now taking values in $[0,1]$.
In these coordinates the IPM system \eqref{eq:ipm} reads
\begin{align}\begin{split}\label{eq:ipm_with_s}
    \pa_t s +u \cdot \nabla s =&0,\\
    \divv u =&0,\\
    u=&-\nabla \Pi +s e_2,\end{split},
\end{align}
see Appendix \ref{sec:otto_relaxation_appendix}.

The starting point (Step 1) of Otto's relaxation is the vital fact that, when formulated in Lagrangian coordinates, IPM can be seen as a gradient flow with respect to the potential energy 
$$
E[\Phi]=-\int s(x,0) \Phi(x)\cdot e_2
$$
on the manifold
\begin{align*}
        M_0=\{ \Phi \text{ one-to-one and onto, smooth, volume preserving maps}\}.
    \end{align*}
More precisely, if $(s,u,\Pi)$ is a solution of \eqref{eq:ipm_with_s} then the flow $\Phi(x,t)$ induced by $u$ satisfies
\begin{align}\label{eq:gradient_flow}
        \int \pa_t\Phi(\cdot,t)\cdot w =- dE[\Phi(\cdot,t)]w,\quad \forall w\in T_{\Phi(\cdot,t)}M_0,
    \end{align}
where $dE[\Phi]w$ is the Fr\'echet derivative of the functional $E$ at the point $\Phi\in M_0$ in the direction 
    \begin{align*}
        w\in T_\Phi M_0=\{ w \text{ smooth and such that } \nabla \cdot \left(w\circ \Phi^{-1}\right)=0\}.
    \end{align*}

Fast-forwarding a bit, the next steps of Otto consist of the introduction of a time discretization with stepsize $h>0$ in form of a minimizing movements scheme (Step 2), the extension of the underlying manifold $M_0$ to its $L^2$-closure in order to turn the potentially ill-posed discrete variational problems emanating from Step 2 to well-posed ones (Step 3), and a translation of the now existing sequence of minimizers back to Eulerian coordinates (Step 4). At this point there exists a sequence of functions $\theta^{(k)}$ corresponding to $s(\cdot,t)$ at time $t=kh$, but of course potentially on a coarse grained or ``locally averaged'' level, which is characterized by the following JKO scheme: $\theta^{(0)}=s(\cdot,0)$, and given $\theta^{(k)}$, $\theta^{(k+1)}$ is the minimizer in $K$ of
   \begin{align}\label{eq:eulerian_JKO_scheme}
   \frac{1}{2}\text{dist}^2(\theta^{(k)},\theta)+\frac{1}{2}\text{dist}^2(1-\theta^{(k)},1-\theta)
   -h\int \theta(x)x_2
   \end{align}
   where the set $K$ consists of measurable $\theta$ taking values in $[0,1]$ and such that $\int \theta=\int s(x,0)$, and 
   $\text{dist}^2(\theta_0,\theta_1)$, for $\theta_0,\theta_1\in K$, is the $L^2$-Wasserstein distance
   \begin{align*}
       \text{dist}^2(\theta_0,\theta_1)=\inf_{\Phi\in I(\theta_0,\theta_1)}\int \theta_0(x)|\Phi(x)-x|^2dx
   \end{align*}
   with
   \begin{align*}
       I(\theta_0,\theta_1)=\{\Phi\, : \, \int \theta_1(y)\zeta(y)dy=\int \theta_0(x)\zeta(\Phi(x))dx \quad \forall \zeta \in \cC^0_0\}. 
   \end{align*}
Notice that this indeed is a relaxation of the original problem since the 
densities are no longer taking values in $\set{0,1}$ and the transport maps
are not necessarily injective.

The fifth and last step consists of passing to the limit $h\rightarrow 0$ whenever this is possible. In \cite{otto1999}, Otto proved that this is the case for the unstable flat situation \begin{align*}s(x,0)=\left\{\begin{array}{cc} 0 & x_2>0\\ 1 & x_2<0\end{array}\right.,\end{align*}
and that the limit of $\theta_h$ defined by 
\[
\theta_h(x,t):=\theta^{(k)}(x),\quad t\in[kh,(k+1)h)
\]
is the unique entropy solution of the conservation law
\begin{align*}
    \pa_t\theta +\pa_{x_2}\left(\theta(1-\theta)\right)=0.
\end{align*}
For a different proof of this statement we refer to the work of Gigli and Otto \cite{GigliOtto} which in particular also contains a further examination of the relation between the minimizing movements scheme and the entropy condition.

In fact, it was conjectured by Otto \cite{otto_conjecture} that the described scheme, if it converges,  should also lead to an entropy solution of the macroscopic IPM equation in the general, non-flat case. We refer to  Section \ref{sec:main_result} for the definition of entropy solutions.

In the rest of this section we sketch how at least formally system \eqref{eq:subequation_inintroduction}, or rather its equivalent reformulation in terms of $s(x,t)$, arises from the JKO-characterization \eqref{eq:eulerian_JKO_scheme} of the discrete functions $\theta^{(k)}$ when \emph{assuming} suitable convergence. Our presentation here, as well as in Appendix \ref{sec:otto_relaxation_appendix} which contains some more details, is devoted  to convey that  the scheme indeed leads to the macroscopic IPM equation, rather than in providing a rigorous proof which we defer to future work. A similar computation was derived by Otto \cite{otto_conjecture}.

Fix $t$ and denote for simplicity $\theta^0:=\theta_h(t)$, $\theta^1:=\theta_h(t+h)$. Furthermore, let  $\Phi^{h}$ denote the transport map corresponding to $\text{dist}^2(\theta^0,\theta^1)$ and $\bar{\Phi}^{h}$ the transport
map corresponding to $\text{dist}^2(1-\theta^{0},1-\theta^{1})$. Then
it can be shown that there are functions $a^{h}$, $\bar{a}^{h}$
such that
\begin{align*}
    \Phi^h(x)&=x+\left(\nabla a^{h}\circ\Phi^h\right)(x),\\
    \bar{\Phi}^{h}(x)&= x+\left(\nabla \bar{a}^{h} \circ \bar{\Phi}^{h}\right)(x).
\end{align*}
This in fact is a consequence of Brenier's Theorem \cite{brenier_optimal_transport}, still an argument is also provided in Appendix \ref{sec:otto_relaxation_appendix}.

Moreover, it can be deduced from first variations of the functional \eqref{eq:eulerian_JKO_scheme} that
\begin{align}\label{eq:relation_between_a}
a^{h}-\bar{a}^{h}=h x_2.
\end{align}
Now, we write  $a^{h}=h p^{h}$, $\bar{a}^h=h \bar{p}^{h}$ and make the strong assumption that the introduced functions $p^{h},\bar{p}^{h}$ have a well defined $\cC^2$ limit denoted by $p$, $\bar p$. Moreover, we also assume that $\theta_h(t,x)$ is converging in a strong enough sense and denote the limit function by $\theta(t,x)$.

If this is the case we can pass to the limit $h\rightarrow 0$ and obtain, cf. Appendix \ref{sec:otto_relaxation_appendix},
\begin{align}\label{diezbis}
    \pa_t \theta &= -\divv(\theta \nabla p),\\\label{oncebis}
  \pa_t\theta&=  \Delta \bar{p}-\divv(\theta \nabla \bar{p}).
\end{align}

Now \eqref{eq:relation_between_a} yields $p=\bar{p}+x_2$. Thus \eqref{diezbis}, \eqref{oncebis} imply that
\begin{align}\label{veinticuatrobis}
    \Delta \bar{p} =\divv ((\nabla \bar p-\nabla p)) \theta)=-\partial_{x_2}\theta.
\end{align}

Therefore, from \eqref{oncebis} and \eqref{veinticuatrobis}, we deduce
\begin{align*}
    \pa_t\theta=& -\pa_{x_2}\theta -\divv (\nabla \bar{p}\theta)\\
    =& -\pa_{x_2}\theta -\divv (\left(\nabla\bar{p}+\theta e_2\right)\theta)+\divv (\theta^2e_2)
\end{align*}
To finish we define $u=\nabla \bar{p}+\theta e_2$, which clearly satisfies $\divv u=0$, to get
\begin{align*}
    \pa_t \theta +u\cdot \nabla  \theta +\pa_{x_2}\theta -2\theta \pa_{x_2}\theta=&0,\\
    u=&\nabla \bar{p}+\theta e_2,\\
    \divv  u =&0.
\end{align*}
Undoying the change of coordinates from the beginning, i.e., considering 
\[
\rho(t,x)=1-2s(x,2t),
\]
one obtains \eqref{eq:subequation_inintroduction}. As said some more details can be found in Appendix \ref{sec:otto_relaxation_appendix}.

\subsubsection{Transfer to subsolutions}\label{sec:transfer_to_subsolutions}
Now we give an alternative derivation of the macroscopic system \eqref{eq:subequation_inintroduction}, taking a different route after step 1 of Otto's relaxation. i.e., the starting point is again the gradient flow structure of IPM saying that solutions of \eqref{eq:ipm} seek to maximize the dissipation of potential energy at every time instance. However, at this point we do not care in which precise sense the dissipation is maximized (in Lagrangian coordinates with respect to the $L^2$-metric on the manifold of area preserving diffeomorphisms). We instead simply extend the principle of maximal energy dissipation for solutions of \eqref{eq:ipm} to its relaxation given in Theorem \ref{thm:CI}, i.e. we seek to investigate also subsolutions that decrease the potential energy at every time instant as much as possible.

Suppose that $(\rho, v, m)$ is a  subsolution in the sense of Definition \ref{def:subsolutions}. We define its associated relative potential energy 
\begin{align}\label{eq:relative_potential_energy}
    E_{rel}(t):=\int_{\T\times\R}( \rho(t,x)-\rho_0(x)) x_2\:dx
\end{align}
and, for now formally, compute
\begin{align}\label{eq:energy_dissipation}
    \partial_t E_{rel}(t)=-\int_{\T\times\R}x_2\divv  m(t,x) \:dx=\int_{\T\times\R} m_2(t,x)\:dx.
\end{align}
Moreover, similar to \eqref{eq:ansatz_for_m_in_flat_case}, condition \eqref{eq:pointwise_constraint_relaxed} implies 
\begin{align*}
    m=\rho v-\frac{1-\rho^2}{2}e_2+\frac{1-\rho^2}{2}\xi
\end{align*}
almost everywhere for some $\xi:[0,T)\times \T\times\R\rightarrow\R^2$ satisfying $\abs{\xi}<1$. Plugging this into \eqref{eq:energy_dissipation} one deduces
\begin{align*}
    \partial_tE_{rel}(t)=\int_{\T\times\R}\rho v_2-(1-\rho^2)\frac{1-\xi_2}{2}\:dx.
\end{align*}
Hence considering $\rho(t,\cdot)$, and therefore also $v(t,\cdot)$, cf. Section \ref{sec:velocity} below, to be given, one easily sees that the energy dissipation at time $t$ is maximized in the closure of all admissible $\xi$ with the choice $\xi(t,x)=-e_2$.

Hence choosing constantly $\xi=-e_2$, and therefore
\begin{align}\label{eq:maximal_dissipating_choice_for_m}
    m=\rho v-(1-\rho^2)e_2
\end{align}
we deduce that (non-strict) subsolutions that maximize at each time instant the dissipation of potential energy are characterized as solutions of
\begin{align}\begin{split}\label{eq:subequation_after_derivation}
    \partial_t\rho+\divv\left(\rho v-(1-\rho^2)e_2\right)&=0,\\
    \divv v&=0,\\
    v&=-\nabla p-\rho e_2.
    \end{split}
\end{align}

The above formal computation in \eqref{eq:energy_dissipation} can be made rigorous under mild decay assumptions, as for instance shown in Appendix \ref{sec:rigorous_dissipation}. Here however, we like to state some further remarks.

First of all we emphasize that by choosing $m$ as in \eqref{eq:maximal_dissipating_choice_for_m} we do not obtain a subsolution in the sense of Definition \ref{def:subsolutions}, since \eqref{eq:pointwise_constraint_relaxed} holds only in a non-strict sense, thus we speak about a non-strict subsolution. By considering instead 
\begin{equation}\label{eq:choice_of_m_arbitrarily_close}
    m=\rho v-\mu (1-\rho^2)e_2,
\end{equation}
i.e. $\xi=(1-2\mu)e_2$ with $\mu$ arbitrarily close to $1$, but $\mu <1$, one obtains strict subsolutions, and hence actual mixing solutions via Theorem \ref{thm:CI}, arbitrarily close to the non-strict ones with maximal energy dissipation. However, in the remaining part of the paper we will solve \eqref{eq:subequation_after_derivation} as the outstanding case and remark that a similar analysis leads to a subsolution corresponding to the system with $m$ given by \eqref{eq:choice_of_m_arbitrarily_close}, cf. also Remark \ref{rem:remark_after_mainthm}.  

Moreover, we like to point out that in the flat case, where $v=0$, system \eqref{eq:subequation_after_derivation} is exactly the hyperbolic conservation law found in \cite{Sz-ipm}, whose entropy solution corresponds to the maximum speed of expansion of the mixing zone, cf. Section \ref{sec:examples_of_subsolutions}.

Furthermore, we remark that given a strict subsolution $(\rho,v,m)$ with relative potential energy $E_{rel}(t)$ defined in \eqref{eq:relative_potential_energy} one obtains infinitely many mixing solutions $(\rho_{sol},v_{sol})$ as in Theorem \ref{thm:CI} with the additional property that their relative potential energy at almost every time $t$ is arbitrarily close to $E_{rel}(t)$, see Lemma \ref{lem:CI_regarding_potential_energy}. In that sense there also exist actual mixing solutions with potential energy decay arbitrarily close to the maximal decay for subsolutions characterized by \eqref{eq:subequation_after_derivation}.

\tikzset{every picture/.style={line width=0.75pt}} 
\begin{figure}[h!]
    \centering
    \begin{tikzpicture}[x=0.75pt,y=0.75pt,yscale=-1,xscale=1]

\draw    (58,418) -- (162,418) -- (162,472) -- (58,472) -- cycle  ;
\draw (110,445) node   [align=left] {\begin{minipage}[lt]{68pt}\setlength\topsep{0pt}
\begin{center}
macroscopic IPM \eqref{eq:subequation_inintroduction}
\end{center}

\end{minipage}};
\draw    (38,208) -- (182,208) -- (182,258) -- (38,258) -- cycle  ;
\draw (110,233) node   [align=left] {\begin{minipage}[lt]{95.2pt}\setlength\topsep{0pt}
\begin{center}
IPM subsolutions,\\Theorem \ref{thm:CI}\\
\end{center}

\end{minipage}};
\draw    (58,8) -- (162,8) -- (162,62) -- (58,62) -- cycle  ;
\draw (110,35) node   [align=left] {\begin{minipage}[lt]{68pt}\setlength\topsep{0pt}
\begin{center}
IPM  \eqref{eq:ipm}
\end{center}

\end{minipage}};
\draw (110,335) node   [align=left] {\begin{minipage}[lt]{150pt}\setlength\topsep{0pt}
\begin{center}
imposing maximal dissipation of pot. energy,\\Section \ref{sec:transfer_to_subsolutions}
\end{center}

\end{minipage}};
\draw (110,127.5) node   [align=left] {\begin{minipage}[lt]{150pt}\setlength\topsep{0pt}
\begin{center}
reformulation as differential inclusion, \cite{Cordoba_Faraco_Gancedo,Sz-ipm}, and relaxation, \cite{Sz-ipm}
\end{center}

\end{minipage}};
\draw    (348,418) -- (542,418) -- (542,472) -- (348,472) -- cycle  ;
\draw (445,445) node   [align=left] {\begin{minipage}[lt]{129.2pt}\setlength\topsep{0pt}
\begin{center}
well-posed variational problems \cite[(2.10)]{otto1999}
\end{center}

\end{minipage}};
\draw (260,480) node   [align=left] {\begin{minipage}[lt]{108.8pt}\setlength\topsep{0pt}
\begin{center}
Eulerian coordinates and $h \rightarrow 0$,\\\cite{otto1999}, Steps 4-5
\end{center}

\end{minipage}};
\draw    (348,8) -- (542,8) -- (542,62) -- (348,62) -- cycle  ;
\draw (445,35) node   [align=left] {\begin{minipage}[lt]{129.2pt}\setlength\topsep{0pt}
\begin{center}
gradient flow for potential energy, \eqref{eq:gradient_flow}
\end{center}

\end{minipage}};
\draw (255,60) node   [align=left] {\begin{minipage}[lt]{129.2pt}\setlength\topsep{0pt}
\begin{center}
Lagrangian coordinates,\\\cite{otto1999}, Step 1
\end{center}

\end{minipage}};
\draw (445,235) node   [align=left] {\begin{minipage}[lt]{156.4pt}\setlength\topsep{0pt}
\begin{center}
time step h and relaxation of discrete variational problems,\\\cite{otto1999}, Steps 2-3
\end{center}

\end{minipage}};
\draw    (110,258) -- (110,306)(110,366) -- (110,415) ;
\draw [shift={(110,418)}, rotate = 270] [fill={rgb, 255:red, 0; green, 0; blue, 0 }  ][line width=0.08]  [draw opacity=0] (8.93,-4.29) -- (0,0) -- (8.93,4.29) -- cycle    ;
\draw    (110,62) -- (110,98)(110,161) -- (110,205) ;
\draw [shift={(110,208)}, rotate = 270] [fill={rgb, 255:red, 0; green, 0; blue, 0 }  ][line width=0.08]  [draw opacity=0] (8.93,-4.29) -- (0,0) -- (8.93,4.29) -- cycle    ;
\draw    (348,445) -- (165,445) ;
\draw [shift={(162,445)}, rotate = 360] [fill={rgb, 255:red, 0; green, 0; blue, 0 }  ][line width=0.08]  [draw opacity=0] (8.93,-4.29) -- (0,0) -- (8.93,4.29) -- cycle    ;
\draw    (445,62) -- (445,203)(445,267) -- (445,415) ;
\draw [shift={(445,418)}, rotate = 270] [fill={rgb, 255:red, 0; green, 0; blue, 0 }  ][line width=0.08]  [draw opacity=0] (8.93,-4.29) -- (0,0) -- (8.93,4.29) -- cycle    ;
\draw    (162,35) -- (345,35) ;
\draw [shift={(348,35)}, rotate = 180] [fill={rgb, 255:red, 0; green, 0; blue, 0 }  ][line width=0.08]  [draw opacity=0] (8.93,-4.29) -- (0,0) -- (8.93,4.29) -- cycle    ;

\end{tikzpicture}
    \caption{Relaxation of IPM in Eulerian coordinates via subsolutions on the left and in Lagrangian coordinates via minimizing movements on the right.}
    \label{fig:epic_diagram}
\end{figure}

\subsubsection{Comparison to selection criteria in related problems}\label{sec:comparison}
As mentioned in Section \ref{sec:subsolution_selection_problem} the selection of a meaningful subsolution is a general problem when studying hydrodynamic instabilities via differential inclusions. We briefly give an overview of previously applied selection criteria.

In the case of a perfectly flat interface the selection typically is done by reducing the subsolution system to a one-dimensional hyperbolic conservation law and picking the unique entropy solution as a natural candidate. This has been done in the context of the Kelvin-Helmholtz instability for the Euler equations \cite{Sz-KH}, the Rayleigh-Taylor instability for the inhomogeneous Euler equations \cite{GKSz}, and as discussed in all detail above for the flat unstable Muskat problem in IPM \cite{Sz-ipm}.

Another approach, selecting the subsolution that at initial time maximizes the total energy dissipation, has been applied in the context of the non-flat Kelvin-Helmholtz instability \cite{Mengual_Sz_vortex_sheet} within the class of all subsolutions with vorticity concentrated on a finite number of sheets, and thereafter in the class of one-dimensional self-similar subsolutions emanating from the flat Rayleigh-Taylor instability modelled by the Euler equations in Boussinesq approximation \cite{GKBou}. This strategy has been motivated by Dafermos' entropy rate admissibility criterion \cite{Dafermos_entropy_rate}, which has also been investigated in \cite{Chiodaroli_Kreml_energy_dissipation,Feireisl} for convex
integration solutions of the compressible Euler equations. In view of Section \ref{sec:transfer_to_subsolutions} also the selection criterion considered in the present paper falls into that category. However, in contrast to  \cite{Mengual_Sz_vortex_sheet,GKBou} the selection applies among all possible subsolutions (with certain natural decay at infinity) and not only within a special subclass, and it applies at all times instead of only the initial time. 

Another way to select subsolutions globally in time has been studied in \cite{GKHirsch} in the context of the flat Rayleigh-Taylor instability for the Euler equations in Boussinesq approximation. Similar as in Section \ref{sec:transfer_to_subsolutions} above, the underlying geometric principle of the equation, in that case the least action principle, has been imposed on the level of subsolutions leading to a degenerate elliptic variational problem that turns out to be formally equivalent to the direct relaxation of the least action principle by Brenier \cite{Brenier89}. However, solutions obtained from this relaxation conserve the total energy, which is inconsistent with anomalous energy dissipation present in turbulent regimes. In view of that, in \cite{GKHirsch} an additional term, responsible for energy dissipation, but subject to certain choices, has been added in the variational problem. In contrast the here considered relaxation of IPM is not relying on any comparable choices. 

\section{The main result}\label{sec:main_result}
According to the previous section we consider on $\T\times\R$ the system
\begin{align}\begin{split}\label{eq:subequation}
\partial_t \rho +\divv(\rho v +\rho^2e_2)&=0,\\
\divv v&=0,\\
v&=-\nabla p-\rho e_2
\end{split}
\end{align}
with initial data \eqref{eq:initial_data}, i.e.
\begin{equation*}
\rho_0(x)=\begin{cases}
+1,&x_2>\gamma_0(x_1),\\
-1,&x_2<\gamma_0(x_1)
\end{cases}
\end{equation*}
for a sufficiently regular function $\gamma_0:\T\rightarrow\R$. In fact we here consider the case of a real analytic initial interface. For completeness we also state the notion of a general weak solution to system \eqref{eq:subequation}.
\begin{definition}\label{def:solution_ipmsub}
A pair $\rho\in L^\infty((0,T)\times\T\times\R)$, $v\in L^\infty(0,T;L^2(\T\times\R;\R^2))$ is a solution of \eqref{eq:subequation}, \eqref{eq:initial_data} provided for any $\varphi\in\cC_c^\infty([0,T)\times\T\times\R)$ there holds
\begin{align*}
    \int_0^T\int_{\T\times\R}\rho\partial_t\varphi+(\rho v+\rho^2 e_2)\cdot\nabla \varphi\:dx\:dt+\int_{\T\times\R}\rho_0\varphi(0,\cdot)\:dx=0,\\
    \int_0^T\int_{\T\times\R}v\cdot \nabla\varphi\:dx\:dt=0,\\
    \int_0^T\int_{\T\times\R}(v+\rho e_2)\cdot\nabla^\perp \varphi\:dx\:dt=0.
\end{align*}
\end{definition}

\begin{theorem}\label{thm:main}
Let $\gamma_0:\T\rightarrow\R$ be real analytic. Then the initial value problem \eqref{eq:subequation}, \eqref{eq:initial_data} has a local in time solution with the following properties
\begin{enumerate}[(i)]
    \item \label{eq:property_i_mainthm}$\rho$ and $v$ are continuous on $[0,T)\times\T\times\R\setminus\set{(0,x_1,\gamma_0(x_1)):x_1\in\T}$,
    \item \label{eq:property_ii_mainthm}at positive times $\rho(t,\cdot)$ is Lipschitz continuous, $v(t,\cdot)$ is log-Lischitz continuous with
    \begin{gather}\label{eq:Lipschitz_bound_rho}
        \norm{\nabla\rho(t,\cdot)}_{L^\infty(\T\times\R)}\leq C_0t^{-1},\\
        \abs{v(t,x)-v(t,x')}\leq C_0t^{-1}\abs{(x-x')\log\abs{x-x'}}
    \end{gather}
    for $t\in(0,T)$, $x,x'\in\T\times\R$, $\abs{x-x'}\leq 1/2$ and a constant $C_0>0$ depending on $\gamma_0$,
    \item \label{eq:property_iii_mainthm}
    for $t\in(0,T)$ there exist two real analytic curves $\gamma_t(\cdot,\pm 1):\T\rightarrow\R$ such that $\rho(t,x)=1$ whenever $x_2\geq\gamma_t(x_1,1)$ and $\rho(t,x)=-1$ whenever $x_2\leq\gamma_t(x_1,-1)$. Moreover, $\rho(t,\cdot)$ maps the remaining set into $(-1,1)$. Also there the level sets $\Gamma_{t}(h):=\set{x\in\T\times\R:\rho(t,x)=h}$, $h\in(-1,1)$ are given by graphs of real analytic functions $\gamma_t(\cdot,h):\T\rightarrow\R$. Furthermore, the joint map $[0,T)\times\T\times[-1,1]\rightarrow \R$, $(t,x_1,h)\mapsto \gamma_t(x_1,h)$ belongs to the space $\cC^1([0,T);\cC^1(\T\times[-1,1]))$ and there exists a real analytic function $s_0:\T\rightarrow\R$ such that
    \begin{equation}\label{eq:expansion_of_level_sets}
        \gamma_t(x_1,h)=\gamma_0(x_1)+t(2h+s_0(x_1))+o(t)
    \end{equation}
    with respect to $\norm{\cdot}_{\cC^{1}(\T\times[-1,1])}$ as $t\rightarrow 0$,
    \item\label{eq:property_iv_mainthm} for any locally Lipschitz continuous $\eta:\R\rightarrow\R$ there holds the balance
    \begin{equation}\label{eq:entropy_with_equality}
        \partial_t (\eta(\rho))+\divv\big( \eta(\rho)v+Q(\rho)e_2\big)=0,
    \end{equation}
    with initial data $\eta(\rho)(0,\cdot)=\eta(\rho_0)$ and flux $Q(\rho):=\int_0^\rho 2\eta'(s)s\:ds$.
\end{enumerate}
\end{theorem}
\begin{remark}\label{rem:remark_after_mainthm} a) In fact the function $s_0:\T\rightarrow\R$ appearing in \eqref{eq:expansion_of_level_sets} is precisely the normal part of the initial velocity when evaluated in $(x_1,\gamma_0(x_1))$. See Section \ref{sec:velocity}, in particular equation \eqref{eq:initial_normal_speed_sec5}, for the definition and further discussion.

b) Note that \ref{eq:property_iii_mainthm} implies that $\rho$ is piecewise $\cC^1$ with the exceptional set given by $\set{(t,x_1,\gamma_t(x_1,\pm 1):t\in[0,T),~x_1\in\T}$.

c) Equation \eqref{eq:entropy_with_equality} is apriori understood in analogy to Definition \ref{def:solution_ipmsub}, i.e. in a distributional sense. However, given the regularity of $\rho$ and $v$ it in fact holds pointwise almost everywhere on $(0,T)\times\T\times\R$, cf. Section \ref{sec:justification}. 

d) Since convex functions are locally Lipschitz, the balance \eqref{eq:entropy_with_equality} in particular states that $\rho$ is an entropy solution for the conservation law $\partial_t\rho+\divv\big(\rho v+\rho^2e_2)=0$, cf. Definition \ref{def:entropy_solution} below. 

e) We notice that, for an analytic initial interface,  the Muskat equation \eqref{eq:muskat_equation} can be solved for short time in order to find a solution to the macroscopic IPM system \eqref{eq:subequation}, which at the same time is also a solution for IPM (see \cite{CCFGM} and in the case of the vortex-sheet problem \cite{naive}).  However, this solution is not an entropy solution. Moreover, piece-wise constant solutions of \eqref{eq:subequation} also could be constructed but again they would not be entropy solutions. 

f) As discussed earlier in Section \ref{sec:transfer_to_subsolutions} the solution $(\rho,v)$ given by Theorem \ref{thm:main} induces only a non-strict subsolution by setting  $m:=\rho v-(1-\rho^2)e_2$. However, an analoguous existence statement remains true when replacing the first equation of \eqref{eq:subequation} by
\begin{align*}
    \partial_t\rho+\divv(\rho v+\mu\rho^2e_2)=0
\end{align*}
corresponding to a choice of $m$ as in \eqref{eq:choice_of_m_arbitrarily_close} and thus to strict subsolutions when $\mu<1$. This can be seen for instance by rescaling time and considering the nonlocal velocity field $\mu^{-1}v$ in Sections \ref{sec:level_set_formulation}, \ref{sec:solving}.

g) Notice that \ref{eq:property_iii_mainthm} describes precisely the mixing zone $\mathscr U$ of the subsolution, cf. Definition \ref{def:subsolutions}, where the corresponding solutions develop a mixing behaviour. In particular, from \eqref{eq:expansion_of_level_sets} one can deduce the initial growth of the mixing zone, which is linear in time. When combined with \cite{Castro_Faraco_Mengual}, it also implies the observed degraded mixing property of solutions (the closer to the upper boundary, the bigger the volume fraction of the heavier fluid).
In particular by letting $\mu$ tend to one, our method predicts a unique mixing zone selected by maximal potential energy dissipation which can be compared with experiments, as opposed to subsolutions where the mixing zone depends on an a priori ansatz.

h) The time of existence $T>0$ of the found solution depends on how good $\gamma_0$ can be extended holomorphically onto a complex strip, see e.g. Lemma \ref{lem:initial_speed}. In addition $T$ is capped by $1$. While the latter is an artifical bound making our proof of existence at some points slightly less technical, the former dependence is naturally appearing in proofs relying on Cauchy-Kovalevskaya Theorems. The question regarding a global in time solution, may it be as a general entropy solution or as a solution of the level set formulation introduced in Section \ref{sec:level_set_formulation}, is open. 

i) The choice of the periodic infinite strip $\T\times\R$ as our spatial domain seemed to us as the least technical choice. Compared to the whole plane $\R^2$ one does not need to speak about decay/flatness at $x_1\rightarrow\pm\infty$, still we believe that our approach can be adapted to that setting. The same is true for the bounded periodic domain $\T\times(0,1)$ where the necessary estimates for the Biot-Savart kernel, cf. Lemma \ref{lem:good_sets_for_K_2}, have to be derived on a more abstract level. However, the situation in a bounded domain with vertical boundaries is more delicate and not in the scope of this paper.
\end{remark}

For completeness we add in the following the notion of an entropy solution for equation \eqref{eq:subequation}. Note that \eqref{eq:subequation} is a nonlocal hyperbolic conservation law. As is common for such equations, cf. e.g. \cite{amadori2012,amorim2012,betancourt2011,blandin2016,colomboR2012}, the notion of an entropy solution is the one for the corresponding local conservation law where the otherwise nonlocal velocity field is considered as a fixed local one:
\begin{definition}[Entropy solution]\label{def:entropy_solution}
    A solution $(\rho,v)$ in the sense of Definition \ref{def:solution_ipmsub} is called an entropy solution provided for any $\varphi\in \cC^\infty_c([0,T)\times\T\times\R)$, $\varphi\geq 0$ and any convex $\eta:\R\rightarrow \R$ with induced flux $Q(\rho):=\int_0^\rho 2\eta'(s)s\:ds$ there holds
    \begin{align*}
        \int_0^T\int_{\T\times\R}\eta(\rho)\partial_t\varphi+(\eta(\rho)v+Q(\rho)e_2)\cdot\nabla\varphi\:dx\:dt+\int_{\T\times\R}\eta(\rho_0)\varphi(0,\cdot)\:dx\geq 0.
    \end{align*}
\end{definition}
We remark that typically the set of $\eta$ for which the stated imbalance is required to hold is taken to be a strict subset of all convex functions, such as for instance the Kru\v{z}kov family $\set{r\mapsto\abs{r-c}:c\in\R}$, \cite{kruzkov1970}, see also \cite{dafermos_book}. Since our solution anyhow satisfies the stronger condition \ref{eq:property_iv_mainthm}, we refrain at this point from restricting the set of entropies.

In any case, due to the nature of the nonlocality of our velocity field, which is a zero order singular integral operator with respect to the density $\rho$ (see Section \ref{sec:velocity}) the uniqueness of the found entropy solution remains open. 

\section{Level set formulation}\label{sec:level_set_formulation}

We begin our investigation with a look at the illustrative example of a perfectly flat initial interface $\gamma_0(x_1)=0$ (Section \ref{sec:flat_interface}) and some known facts concerning the nonlocal velocity field $v$, in particular at initial time, in the non-flat case (Section \ref{sec:velocity}).
Thereafter, with the beginning of Section \ref{sec:ansatz_1}, we will reformulate problem \eqref{eq:subequation}, \eqref{eq:initial_data} as a suitable fixed point problem.

\subsection{The flat interface}\label{sec:flat_interface}
In the prefectly flat case, $\gamma_0=0$, a $x_1$-independent solution of equation \eqref{eq:subequation} is obtained by observing that $v=0$ and solving the Riemann problem for Burgers' equation
\begin{align*}
    \partial_t\rho+\partial_{x_2}(\rho^2)=0,\quad \rho(0,x_2)=\sign(x_2).
\end{align*}
The unique entropy solution is Lipschitz continuous at positive times and explicitly given by 
\[
\rho(t,x)=\begin{cases}
1,&x_2>2t,\\
\frac{x_2}{2t},&\abs{x_2}\leq 2t,\\
-1,&x_2<-2t.
\end{cases}
\]
As discussed earlier, cf. Section \ref{sec:examples_of_subsolutions}, this solution bounds the mixing zone in the class of all one-dimensional IPM-subsolutions.

However, in rescaled coordinates $y\mapsto x$, $x=(y_1,ty_2)$ the solution is given by the stationary profile
\begin{align}\label{eq:reformulation_rho_burgers}
    \rho(t,y_1,ty_2)=\phi_0(y):=\begin{cases}
1,&y_2>2,\\
\frac{1}{2}y_2,&\abs{y_2}\leq 2,\\
-1,&y_2<-2,
\end{cases}
\end{align}
or in other words the level sets $\rho(t,\cdot)^{-1}(\{h\})$, $h\in(-1,1)$ are given by flat lines $\set{x:x_2=2ht}$ that as time evolves are pulled apart with speed $2h$.

Of course these are simple reformulations, but a key point in our analysis is an appropriate extension of this principle to the general, non-flat case where the velocity field does not vanish. This will be done by keeping the profile $\phi_0(y)$ on the right-hand side of \eqref{eq:reformulation_rho_burgers} and allowing the transformation $y\mapsto x$ to be of the type $x=(y_1,ty_2+f(t,y))$, i.e. we keep the ``pulling''-term $ty_2$ dealing with the Burgers' term $\partial_{x_2}(\rho^2)$ in the equation and allow the level sets to have a general form reacting to the nonlocal velocity field. The details in terms of induced equations for $f$ are in Sections \ref{sec:ansatz_1} to \ref{sec:ansatz3}.

\subsection{Biot-Savart and the initial velocity field}\label{sec:velocity}

The flat case discussed in the previous subsection is a very special case in the sense that $v=0$ and the resulting equation is local.
In the general case a key feature of both systems, IPM and the relaxation, is the nonlocal relation between the density $\rho$ and the velocity field $v$. More precisely,  the last two equations in \eqref{eq:ipm}, \eqref{eq:subequation} resp., i.e. the incompressibility condition and Darcy's law, can be understood by means of a $0$-order convolution operator. Indeed, taking the curl of Darcy's law one sees that at each time $v(t,\cdot)$ is an incompressible vectorfield with vorticity given by
\begin{equation}\label{eq:vorticity_of_v}
    \partial_{x_1}v_2(t,x)-\partial_{x_2}v_1(t,x)=-\partial_{x_1}\rho(t,x).
\end{equation}
Thus when requiring decay as $\abs{x_2}\rightarrow\infty$ the velocity field $v$ is, at least in the case of our interest, uniquely determined in terms of the Biot-Savart operator
\begin{equation}\label{eq:biot_savart}
    v(t,x)=(K*(-\partial_{x_1}\rho(t,\cdot)))(x)=\int_{\T\times\R}K(x-z)(-\partial_{x_1}\rho(t,z))\:dz.
\end{equation}
On $\T\times\R$ the kernel $K$ is given by
\begin{equation}
    K(z):=\frac{1}{4\pi}\frac{(-\sinh(z_2),\sin(z_1))^T}{\cosh(z_2)-\cos(z_1)},
\end{equation}
and as usual, $K$ is the orthogonal gradient of the corresponding Green's function 
\begin{equation}\label{eq:greens_function}
    G(z):=\frac{1}{4\pi}\log(\cosh(z_2)-\cos(z_1)).
\end{equation}

Relation \eqref{eq:biot_savart} has to be interpreted accordingly at initial time $t=0$ due to the fact that $-\partial_{x_1}\rho_0$ is only a measure supported on the interface 
\[
\Gamma_0:=\set{(x_1,\gamma_0(x_1)):x_1\in\T}.
\]
Thus the initial velocity field $v_0(x)$ is the one of a vortex-sheet and therefore discontinuous across the interface.
\begin{lemma}\label{lem:formulas_v0}
The unique square integrable solution of 
\begin{equation}\label{eq:div_curl_system_at_initial_time}
v=-\nabla p-\rho_0e_2,\quad \divv v=0\quad\text{on }\T\times\R
\end{equation}
is given by
\begin{align}\label{eq:initial_velocity_away_from_interface}
    v_0(x)=\int_{\T}K\begin{pmatrix}x_1-z_1\\x_2-\gamma_0(z_1)\end{pmatrix}2\gamma_0'(z_1)\:dz_1
\end{align}
for $x\notin\Gamma_0$, while the one-sided limits at $\Gamma_0$ are given by
\begin{align}\begin{split}\label{eq:initial_velocity_at_interface}
    \lim_{\underset{\pm(y_2-\gamma_0(y_1))>0}{y\rightarrow (x_1,\gamma_0(x_1))}}v_0(y)=p.v.\int_{\T}K\begin{pmatrix}x_1-z_1\\\gamma_0(x_1)-\gamma_0(z_1)    \end{pmatrix}&2\gamma_0'(z_1)\:dz_1\\
    &\mp \frac{\gamma_0'(x_1)}{1+\gamma_0'(x_1)^2}\begin{pmatrix}1\\\gamma_0'(x_1)
    \end{pmatrix}.
    \end{split}
\end{align}
\end{lemma}
\begin{proof}
First of all one can check that the right-hand side of \eqref{eq:initial_velocity_away_from_interface} defines a locally integrable solution of \eqref{eq:div_curl_system_at_initial_time} with exponential decay as $\abs{x_2}\rightarrow\infty$. Thus standard elliptic estimates imply that this is the only solution with these properties. 

In order to compute the one-sided limits we write
\[
K(z)=\frac{1}{2\pi}\frac{z^\perp}{\abs{z}^2}\eta(z_1)+K_{reg}(z),
\]
where $\eta:\T\times\R$ is a smooth periodic cutoff function with $\eta(z_1)=1$ for $\abs{z_1}\leq 1$ and $\eta(z_1)=0$ for $\abs{z_1}\geq 2$, and the regular part $K_{reg}:\T\times\R\rightarrow\R^2$,
\[
K_{reg}(z):=K(z)-\frac{1}{2\pi}\frac{z^\perp}{\abs{z}^2}\eta(z_1)
\]
is smooth. In fact $K_{reg}$ is harmonic where $\eta(z_1)=1$. Furthermore, using complex notation we write $z^\perp/\abs{z}^2=(1/(iz))^*$ where $z^*$ denotes complex conjugation.

Then denoting by $v_{0,reg}$ the contribution from the regular part $K_{reg}$ we have 
\begin{align*}
    v_0(y)-v_{0,reg}(y)&=\left(\frac{1}{2\pi i}\int_{\T}\frac{2\gamma_0'(z_1)}{y-(z_1+i\gamma_0(z_1))}\:dz_1\right)^*
    \\&=-\left(\frac{1}{2\pi i}\int_{\Gamma_0}\frac{1}{\xi-y}\frac{2\gamma_0'(\xi_1)}{1+i\gamma_0'(\xi_1)}\:d\xi\right)^*
\end{align*}
for $y\notin \Gamma_0$. Now taking one sided limits $y\rightarrow x\in\Gamma_0$ expression \eqref{eq:initial_velocity_at_interface} follows from the Sokhotski-Plemelj formula, see \cite{Muskhelishvili}.
\end{proof} 

Formulas \eqref{eq:initial_velocity_at_interface} show that the initial velocity field is still continuous across the interface in the normal direction. Therefore the (not normalized) normal velocity at the interface $s_0:\T\rightarrow\R$, 
\begin{align}\begin{split}\label{eq:initial_normal_speed_sec5}
    s_0(x_1)&:=v_0(x_1,\gamma_0(x_1))\cdot\begin{pmatrix}-\gamma_0'(x_1)\\1\end{pmatrix}
    \\&=p.v.\int_{\T}K\begin{pmatrix}
    x_1-z_1\\\gamma_0(x_1)-\gamma_0(z_1)
    \end{pmatrix}2\gamma_0'(z_1)\:dz_1\cdot\begin{pmatrix}-\gamma_0'(x_1)\\1\end{pmatrix}
    \end{split}
\end{align}
is well-defined. It will play an important role in our further analysis as it dictates the motion of Lagrangian particles at the interface to first order when ignoring the Burgers' term $\partial_{x_2}(\rho^2)$.

\subsection{Rescaling and level set function}\label{sec:ansatz_1}
We now transform problem \eqref{eq:subequation}, \eqref{eq:initial_data} in terms of level sets. The reformulation here is understood on a formal level. We will solve the derived fixed point problem in Section \ref{sec:solving} and aposteriori justify the transformations in Section \ref{sec:justification}.

The starting point is the following ansatz for $\rho$ capturing the effect of the Burger's part described in Section \ref{sec:flat_interface}. 
Assume that there exists $f:[0,T)\times\T\times\R\rightarrow\R$ sufficiently regular with
\begin{equation}\label{eq:initial_data_for_f}
f(0,y)=\gamma_0(y_1)
\end{equation}
and such that for every $t\in(0,T)$, $y_1\in\T$ the map $\R\rightarrow\R$, $y_2\mapsto ty_2+f(t,y_1,y_2)$ is a monotone diffeomorphism.

Then each of the transformations
$X_t:\T\times\R\rightarrow\T\times\R$, $t\in(0,T)$,
\[
X_t(y)=\begin{pmatrix}y_1\\ty_2+f(t,y)\end{pmatrix}
\] 
is a diffeomorphism as well. 

We now seek to find a solution of \eqref{eq:subequation}, \eqref{eq:initial_data} on $[0,T)$ having the property that
\begin{align}\label{eq:ansatz}
\rho(t,X_t(y))=\phi_0(y_2)=\begin{cases}
+1,&y_2\geq 2,\\
\frac{1}{2}y_2,&y_2\in (-2,+2),\\
-1,&y_2\leq -2.
\end{cases}
\end{align}

For $t>0$ we compute
\begin{align}
DX_t(y)&=\begin{pmatrix}
1&0\\
\partial_{y_1}f(t,y)&t+\partial_{y_2}f(t,y)
\end{pmatrix},\label{eq:formuala_DX}\\ 
DX_t(y)^{-1}&=\begin{pmatrix}
1&0\\
\frac{-\partial_{y_1}f(t,y)}{t+\partial_{y_2}f(t,y)}&\frac{1}{t+\partial_{y_2}f(t,y)}
\end{pmatrix},\label{eq:formula_DX_inverse}\\
\nabla\rho(t,X_t(y))&=\frac{1}{2(t+\partial_{y_2}f(t,y))}\begin{pmatrix}
-\partial_{y_1}f(t,y)\\1
\end{pmatrix}\mathbbm{1}_{(-2,2)}(y_2)\label{eq:formula_gradient_rho},
\end{align}
such that the first equation of \eqref{eq:subequation}, when written in non divergence form, under the ansatz \eqref{eq:ansatz} is equivalent to 
\begin{align*}
0=\mathbbm{1}_{(-2,2)}(y_2)\begin{pmatrix}
-\partial_{y_1}f(t,y)\\
1
\end{pmatrix}\cdot\left(\begin{pmatrix}0\\y_2+\partial_tf(t,y)-2\phi_0(y_2))\end{pmatrix}-v(t,X_t(y))\right).
\end{align*}
Since  $2\phi_0(y_2)=\mathbbm{1}_{(-2,2)}(y_2)=y_2$,
expanding the above equation leads to 
\begin{align}\label{eq:equation_for_f}
\partial_tf(t,y)=v(t,y_1,ty_2+f(t,y))\cdot \begin{pmatrix}
-\partial_{y_1}f(t,y)\\1
\end{pmatrix}
\end{align}
for  $(t,y_1,y_2)\in(0,T)\times \T\times(-2,2)$.

Note that in view of \eqref{eq:formula_gradient_rho} the velocity field in \eqref{eq:equation_for_f} is always considered in directions normal to the level sets of $\rho$.

\subsection{Transformation of the velocity field}\label{sec:transformation_v}
For $t>0$ we have that $v(t,\cdot)$ (in all reasonable scenarios) is given by the Biot-Savart law \eqref{eq:biot_savart}, cf. Section \ref{sec:velocity}.

Applying the transformation $X_t(y)$ we compute the velocity field $v(t,y_1,ty_2+f(t,y))=v(t,X_t(y))$ occuring in \eqref{eq:equation_for_f}. First of all formulas \eqref{eq:formuala_DX} and \eqref{eq:formula_gradient_rho} imply
\begin{align*}
v(t,X_t(y))&=-\int_{\T\times\R}K(X_t(y)-z)\partial_{x_1}\rho(t,z)\:dz\\
&=-\int_{\T\times\R}K(X_t(y)-X_t(z))\partial_{x_1}\rho(t,X_t(z))\det DX_t(z)\:dz\\
&=\frac{1}{2}\int_{-2}^2\int_{\T}K(X_t(y)-X_t(z))\partial_{y_1}f(t,z)\:dz_1\:dz_2.
\end{align*}
Next we compute the full right-hand side of \eqref{eq:equation_for_f} and exploit the fact that the velocity field $v(t,X_t(y))$ is only needed in normal directions. More precisely, for $z\neq y$ we have
\begin{align*}
\partial_{y_1}f(t,z)K(&X_t(y)-X_t(z))\cdot\begin{pmatrix}
-\partial_{y_1}f(t,y)\\1
\end{pmatrix}\\
&=\partial_{y_1}f(t,z)\nabla G(X_t(y)-X_t(z))\cdot\begin{pmatrix}
1\\\partial_{y_1}f(t,y)
\end{pmatrix}\\
&=\partial_{y_1}f(t,y)\nabla G(X_t(y)-X_t(z))\cdot\begin{pmatrix}
1\\\partial_{y_1}f(t,z)
\end{pmatrix}\\
&\hspace{30pt}-\partial_1G(X_t(y)-X_t(z))(\partial_{y_1}f(t,y)-\partial_{y_1}f(t,z))\\
&=-\partial_{y_1}f(t,y)\frac{d}{dz_1}\big(G(X_t(y)-X_t(z))\big)\\
&\hspace{30pt}-K_2(X_t(y)-X_t(z))(\partial_{y_1}f(t,y)-\partial_{y_1}f(t,z)).
\end{align*}
Thus after integration we obtain an additional cancelation in the convolution, i.e. there holds
\begin{align}
\begin{split}\label{eq:cancelation}
v(t,&X_t(y))\cdot\begin{pmatrix}
-\partial_{y_1}f(t,y)\\1
\end{pmatrix}\\
&\phantom{=}=-\frac{1}{2}\int_{-2}^2\int_{\T}K_2(X_t(y)-X_t(z))(\partial_{y_1}f(t,y)-\partial_{y_1}f(t,z))\:dz_1\:dz_2.
\end{split}
\end{align}

\subsection{\texorpdfstring{Equation for $f$}{Equation for f}}\label{sec:ansatz3}

Combining equation \eqref{eq:cancelation} with \eqref{eq:equation_for_f} we see that \eqref{eq:subequation} can after our ansatz be written in the closed form
\begin{align}\label{eq:closed_equation_for_f}
\partial_tf(t,y)=-\frac{1}{2}\int_{-2}^2\int_{\T}K_2(\tilde{\Delta} X_t(y,z))(\partial_{y_1}f(t,y)-\partial_{y_1}f(t,z))\:dz_1\:dz_2
\end{align}
where
\begin{align*}
\tilde{\Delta} X_t(y,z):=X_t(y)-X_t(z)=\begin{pmatrix}y_1-z_1\\t(y_2-z_2)+f(t,y)-f(t,z)\end{pmatrix}
\end{align*}
also depends on $f$.
Via translation in $z_1$ equation \eqref{eq:closed_equation_for_f} can also be written as
\begin{align}\label{eq:closed_equation_for_f2}
\partial_tf(t,y)=-\frac{1}{2}\int_{-2}^2\int_{\T}K_2(\Delta X_t(y,z))\Delta \partial_{y_1}f_t(y,z)\:dz_1\:dz_2,
\end{align}
where we have abbreviated
\begin{gather}
\begin{gathered}\label{eq:abbreviation_of_deltas_sec4}
\Delta X_t(y,z):=\begin{pmatrix}z_1\\t(y_2-z_2)+f(t,y_1,y_2)-f_t(t,y_1-z_1,z_2)\end{pmatrix},\\
\Delta \partial_{y_1}f_t(y,z):=\partial_{y_1}f(t,y_1,y_2)-\partial_{y_1}f(t,y_1-z_1,z_2).
\end{gathered}
\end{gather}
The latter form turns out to be more convenient to work with.
\subsection{One more ansatz}
One important assumption in the above derivation is the invertibility of the maps $(X_t)_{t>0}$.
In order to guarantee that we further make the ansatz
\begin{equation}\label{eq:ansatz_for_f}
f(t,y)=\gamma_0(y_1)+ts_0(y_1)+\frac{1}{2}t^{1+\alpha} \eta(t,y),
\end{equation}
where $\alpha\in(0,1)$ and the functions $s_0:\T\rightarrow \R$, $\eta:(0,T)\times\T\times \R\rightarrow \R$ are sufficiently regular. In order to avoid potential confusion we emphasize that the function $\eta$ has nothing to do with an entropy as for example appearing in Definition \ref{def:entropy_solution}. Furthermore, we remark that the particular choice of $\alpha\in(0,1)$ is not important, see Section \ref{sec:conclusion_of_existence} for further discussion.

By this ansatz $f$ satisfies \eqref{eq:initial_data_for_f}, and the desired invertibility can assumed to be true for a small time interval (depending on $\norm{\partial_{y_2}\eta}_{L^\infty}$). 
Moreover, since at $t=0$ there holds $\partial_t f=s_0$, $\partial_{y_1}f=\gamma_0'$, passing formally to the limit on the right-hand side of \eqref{eq:closed_equation_for_f2} one sees that  $s_0$ necessarily is given by 
\begin{equation}\label{eq:expression_s_0}
-\frac{1}{2}\int_{-2}^2\int_{\T}K_2(\Delta X_0(y_1,z_1))\Delta\gamma_0'(y_1,z_1)\:dz_1\:dz_2,
\end{equation}
where
\begin{gather*}
    \Delta X_0(y_1,z_1):=\begin{pmatrix}
z_1\\\gamma_0(y_1)-\gamma_0(y_1-z_1)
\end{pmatrix},\\
\Delta\gamma_0'(y_1,z_1):=\gamma_0'(y_1)-\gamma_0'(y_1-z_1).
\end{gather*}

A quick computation similar to the one in Section \ref{sec:transformation_v} and comparison with \eqref{eq:initial_normal_speed_sec5} shows that the above expression is precisely the normal component of the initial velocity evaluated at $(y_1,\gamma_0(y_1))$, i.e. \eqref{eq:expression_s_0}. This shows that the function $s_0(y_1)$ is indeed forced to be the normal component of $v_0(y_1,\gamma_0(y_1))$. 

Finally we integrate \eqref{eq:closed_equation_for_f2} in time and use \eqref{eq:ansatz_for_f}, \eqref{eq:expression_s_0} in order to deduce that for $f$  to be a solution to  \eqref{eq:closed_equation_for_f2}  $\eta$ must be a solution of the following fixed point problem:
\begin{align}
\begin{split}\label{eq:fixed_point_equation_eta}
\eta(t,y)=-\frac{1}{t^{1+\alpha}}\int_0^t\int_{-2}^2\int_\T K_2(\Delta &X_s(y,z))\Delta \partial_{y_1}f_s(y,z)\\
&-K_2(\Delta X_0(y_1,z_1))\Delta\gamma_0'(y_1,z_1)\:dz_1\:dz_2\:ds.
\end{split}
\end{align}
Note that $\eta$ and $\partial_{y_1}\eta$ enter the right-hand side through \eqref{eq:abbreviation_of_deltas_sec4}, \eqref{eq:ansatz_for_f}.

\section{Existence of a solution for analytic graphs}\label{sec:solving}
Our goal is to show that for a real analytic $\gamma_0:\T\rightarrow\R$ there exists a unique local in time solution $\eta$ of problem \eqref{eq:fixed_point_equation_eta}. The proof relies on the following version of the abstract Cauchy-Kovalevskaya theorem based on the formulation of Nishida \cite{nishida}, see also \cite{nirenberg}.

In order to avoid confusion we emphasize that throughout Section \ref{sec:solving} every symbol $\rho,\rho',\bar{\rho},\rho_0$ denotes a positive constant referring to the size of the domain of analyticity. This is done in analogy to \cite{nishida}. At no time in Section \ref{sec:solving} the density function $\rho(t,x)$, which we seek to construct,  or the initial density $\rho_0(x)$ are mentioned.

\begin{theorem}\label{thm:abstract_CK} Let $(B_\rho)_{\rho\in (0,\rho_0)}$, $\rho_0>0$ be a scale of Banach spaces with $\norm{\cdot}_{\rho'}\leq \norm{\cdot}_{\rho}$ for $0<\rho'< \rho< \rho_0$ and consider the integral equation
\begin{equation}\label{eq:abstract_fixed_point}
u(t)=\frac{1}{a(t)}\int_0^tF(u(s),s)\:ds
\end{equation}
for a given continuous function $a:[0,\infty)\rightarrow\R$ with $a(t)>0$ for $t>0$.
If $F$ is such that
\begin{enumerate}[(i)]
\item\label{cond:F_well_defined} there exist $R>0$, $T>0$ such that for every $0< \rho'<\rho<\rho_0$ the map 
\[
\set{u\in B_\rho:\norm{u}_\rho<R}\times [0,T)\rightarrow B_{\rho'},\quad (u,t)\mapsto F(u,t)
\]
is well-defined and continuous,
\item\label{cond:F_lipschitz} there exists $b:[0,T)\rightarrow [0,\infty)$ continuous such that for any $0<\rho'<\rho<\rho_0$ and all $u,v\in B_\rho$, $\norm{u}_\rho<R$, $\norm{v}_\rho<R$, $t\in[0,T)$ there holds
\[
\norm{F(u,t)-F(v,t)}_{\rho'}\leq \frac{b(t)}{\rho-\rho'}\norm{u-v}_{\rho},
\]
\item\label{cond:F_0_well_defined} $F(0,\cdot)\in L^1(0,T;B_\rho)$ for any $\rho\in (0,\rho_0)$ and there exists $c:[0,T)\rightarrow[0,\infty)$ continuously differentiable on $(0,T)$, continuous on $[0,T)$ with $c(0)=0$, as well as $c'(t)>0$ for $t>0$, such that for all $\rho\in(0,\rho_0)$, $t\in (0,T)$ there holds
\[
\frac{1}{a(t)}\int_0^t\norm{F(0,s)}_{\rho}\:ds\leq \frac{c(t)}{\rho_0-\rho},
\]
\item\label{cond:relation_between_functions} for a constant $K>0$ the functions $a(t)$, $b(t)$, $c(t)$ appearing in \eqref{eq:abstract_fixed_point}, \ref{cond:F_lipschitz}, \ref{cond:F_0_well_defined} satisfy the relation
\begin{equation}\label{eq:relation_between_functions_abc}
\sup_{s\in(0,t)}\abs{\frac{b(s)c(s)}{c'(s)}}\leq K a(t)c(t),~t\in(0,T),
\end{equation}
\end{enumerate}
then there exists a constant $\bar{a}=\bar{a}(K,R)>0$ and a unique $u(t)$ which for any $\rho\in(0,\rho_0)$  
maps the interval $\set{t\in[0,T):c(t)<\bar{a}(\rho_0-\rho)}$ continuously into the $R$-ball of $B_\rho$. Moreover, $u$ satisfies \eqref{eq:abstract_fixed_point} and $\norm{u(t)}_\rho=O\left(\frac{c(t)}{\rho_0-\rho}\right)$ as $t\rightarrow 0$. In particular $u(0)=0$.
\end{theorem}

For the choices $a(t)=1$, $b(t)=c_1$, $c(t)=c_2t$ with some constants $c_1,c_2>0$ the above theorem is the abstract Cauchy-Kovalevskaya theorem in the formulation of Nishida \cite{nishida}. The proof of Theorem \ref{thm:abstract_CK} requires indeed just some minor modifications which are presented in Appendix \ref{sec:CK_proof}. For a related generalization of the abstract Cauchy-Kovalevskaya theorem see also \cite{reissig1,reissig2}.

We will apply Theorem \ref{thm:abstract_CK} in the following situation.
\begin{lemma}\label{lem:choices_of_abc} Let $c_1,c_2>0$ and $\alpha\in (0,1)$. There exist $T=T(\alpha),K=K(c_1,c_2)>0$ such that $a(t):=t^{1+\alpha}$, $b(t):=c_1t^{1+\alpha}\abs{\log t}$, $c(t):=c_2t^{1-\alpha}\abs{\log t}$ satisfy \eqref{eq:relation_between_functions_abc}.
\end{lemma}
\begin{proof}
Consider $T\in(0,1)$ such that
\begin{gather*}
(1-\alpha)\abs{\log t}\geq 2, \quad \abs{\log t}t^\alpha\leq 1,
\end{gather*}
for all $t\in(0,T)$. Then for $0<s<t<T$ there holds
\[
\frac{b(s)c(s)}{c'(s)}=c_1\frac{s^{2+\alpha}\abs{\log s}^2}{(1-\alpha)\abs{\log s}-1}\leq c_1s^{2}\abs{\log s}\leq c_1t^2\abs{\log t}=\frac{c_1}{c_2}a(t)c(t).
\]
Thus \eqref{eq:relation_between_functions_abc} holds true with $K:=c_1c_2^{-1}$.
\end{proof}

\subsection{Banach spaces}\label{sec:banach_spaces}
Set
\begin{equation*}
\Omega_0:=\T\times (-2,2),
\end{equation*}
as well as
\begin{gather*}
U_\rho:=\set{z\in\C:\abs{\im(z)}<\rho},\quad \Omega_\rho:=U_\rho\times (-2,2)
\end{gather*}
for $\rho>0$.

We define the space $B_\rho$ to consist of all continuous functions $\eta:\Omega_0\rightarrow \R$, $y\mapsto\eta(y)$ which satisfy
\begin{enumerate}[(i)]
\item for every $y_2\in(-2,2)$ the function $\eta(\cdot,y_2)$ extends to a holomorphic function $U_\rho\rightarrow\C$ which is again denoted by $\eta(\cdot,y_2)$,
\item the derivative $\partial_{y_2}\eta:\Omega_\rho\rightarrow\C$ exists, is uniformly continuous and $\partial_{y_2}\eta(\cdot,y_2)$ is holomorphic on $U_\rho$ for every $y_2\in(-2,2)$,
\item the norm 
\[
\norm{\eta}_{\rho}:=\norm{\eta}_{L^\infty(\Omega_\rho)}+\norm{\partial_{y_1}\eta}_{L^\infty(\Omega_\rho)}+\norm{\partial_{y_2}\eta}_{L^\infty(\Omega_\rho)}
\]
is finite.
\end{enumerate}
For clarification, the extension in (i) strictly speaking is the extension of the $2\pi$-periodic function $\eta(\cdot,y_2):\R\rightarrow \R$. The extension $U_\rho \rightarrow\C$, $y_1\mapsto \eta(y_1,y_2)$ therefore is periodic in the real part of $y_1$.  Moreover, $\partial_{y_1}\eta$ denotes the complex derivative in the first component, while $\partial_{y_2}\eta$ is the real partial derivative with respect to the second component. Although the two derivatives are of slightly different nature, we still use a gradient notation $\nabla_y\eta:=(\partial_{y_1}\eta,\partial_{y_2}\eta)^T$.

Clearly each $B_\rho$ is a Banach space and $B_{\rho}\subset B_{\rho'}$, $\norm{\cdot}_{\rho'}\leq \norm{\cdot}_{\rho}$ whenever $\rho'<\rho$. Moreover, for the introduced scale of spaces we have the following lemma, which is a direct consequence of Cauchy integral formula for analytic functions.
\begin{lemma}[Cauchy]\label{lem:cauchy} Let $0<\rho'<\rho$ and $\eta\in B_\rho$. Then for $j=1,2$ there holds
\begin{align*}
\norm{\partial_{y_1}\partial_{y_j}\eta}_{L^\infty(\Omega_{\rho'})}\leq \frac{C}{\rho-\rho'}\norm{\eta}_\rho
\end{align*}
for $C=(2\pi)^{-1}$.
\end{lemma}

In particular $\partial_{y_1}\eta$ is, as $\eta$ itself, Lipschitz continuous on $\Omega_0$. This together with the assumed uniform continuity of $\partial_{y_2}\eta$ implies
\begin{lemma}\label{lem:real_C1}
Let $\rho>0$ and $\eta\in B_\rho$. Then $\eta:\Omega_0\rightarrow \R$ extends to $\cC^1(\overline{\Omega}_0)$ and $\eta(\cdot,y_2)$, $\partial_{y_2}\eta(\cdot,y_2)$ are real analytic for each $y_2\in[-2,2]$.
\end{lemma}

Also note that $\partial_{y_2}\partial_{y_1}\eta(y)=\partial_{y_1}\partial_{y_2}\eta(y)$ for $\eta\in B_\rho$, $y\in\Omega_\rho$ for instance by means of Cauchy's integral formula.

\subsection{Notation}\label{sec:notation}

From now on we fix $\alpha\in(0,1)$ and a real analytic initial datum $\gamma_0:\T\rightarrow\R$. Clearly $\gamma_0$ can be extended to a holomorphic function defined on $U_{2\rho_0}$ for some $\rho_0>0$ small. 

Hence all (complex) derivatives are uniformly bounded on $U_{\rho_0}$, e.g. there exist a constant $C_0>0$ such that
\begin{equation}\label{eq:bounds_for_derivative_of_gamma_0}
\norm{\gamma_0'}_{L^\infty(U_{\rho_0})}\leq C_0.
\end{equation}

More generally, in the following $C_0>0$ always denotes a constant depending solely on the $L^\infty(U_{\rho_0})$-norm of a fixed finite amount of derivatives of $\gamma_0$. (A detailed look at the proof reveals that the first five derivatives of $\gamma_0$ are sufficient. However, the precise number is not important.)
In contrast $C>0$ usually denotes a constant not depending on $\gamma_0$. Both constants typically change from line to line. Also we point out that distinguishing $C_0$ from $C$ is not essential for the proof of Theorem \ref{thm:main}.

For a pair $a=(a_1,a_2)\in \R\times \C$ we denote
\begin{align}\label{eq:definition_of_complex_norm}
\abs{a}_*:=\left(\abs{a_1}^2+\abs{a_2}^2\right)^{\frac{1}{2}}=\left(a_1^2+a_2a_2^*\right)^{\frac{1}{2}}.
\end{align}
Moreover, whenever we write $\abs{z_1}$ for $z_1\in\T$ we mean the absolute value of the unique representant of $z_1$ in $[-\pi,\pi)$. In particular we will also use $\abs{a}_*$ for pairs $a\in\T\times\C$.

For any function $g:\Omega_{\rho_0}\rightarrow\C^n$ or $h:U_{\rho_0}\rightarrow \C^n$ we abbreviate
\begin{align}\begin{split}\label{eq:Delta_notation}
\Delta g(y,z)&:=g(y)-g(y_1-z_1,z_2),\\
\Delta h(y_1,z_1)&:=h(y_1)-h(y_1-z_1)
\end{split}
\end{align}
for $y=(y_1,y_2)\in\Omega_{\rho_0}$, $z=(z_1,z_2)\in\Omega_0$ and $y_1\in U_{\rho_0}$, $z_1\in\T$ resp.. In proofs we will most of the time omit the points $(y,z)$ and simply write $\Delta g$, $\Delta h$.

Furthermore, for $t\geq 0$ and $\eta\in B_{\rho_0}$ we define $f_t^\eta:\Omega_{\rho_0}\rightarrow\C$, $X_t^\eta:\Omega_{\rho_0}\rightarrow\C^2$,
\begin{align}\label{eq:definition_of_feta_Xeta}
f_t^\eta(y):=\gamma_0(y_1)+ts_0(y_1)+\frac{1}{2}t^{1+\alpha}\eta(y),\quad X_t^\eta(y):=\begin{pmatrix}
y_1\\ty_2+f_t^\eta(y)
\end{pmatrix}.
\end{align}
The function $s_0:U_{\rho_0}\rightarrow\C$ will be introduced in Lemma \ref{lem:initial_speed} below. At time $t=0$ we simply write $X_0(y_1)$ instead of $X_0^\eta(y)$. The second component of $X_t^\eta(y)$ is denoted by $X_{t,2}^\eta(y)$. There is no need to distinguish the first component, since it is just given by $y_1$.

\subsection{Preliminary lemmas}\label{sec:preliminary_lemmas}
In order to define the function $F$ as a complex extension of the functional appearing in \eqref{eq:fixed_point_equation_eta} we need some preparation.

Recall that the second component $K_2$ of the Biot-Savart kernel on $\T\times \R$ is given by
\begin{align*}
K_2(a)=K_2(a_1,a_2)=\frac{1}{4\pi}\frac{\sin(a_1)}{\cosh(a_2)-\cos(a_1)}.
\end{align*}
Thus for fixed $a_1\in\T$ the canonical extension of $K_2(a_1,\cdot)$ to $a_2\in\C$ is holomorphic on the open set $\set{a_2\in\C:\cosh(a_2)-\cos(a_1)\neq 0}$. We define 
\[
\cU:=\set{a\in\T\times\C:\cosh(a_2)-\cos(a_1)\neq 0}.
\]

\begin{lemma}\label{lem:good_sets_for_K_2} Let $\kappa\in(0,1/2)$. The sets 
\[
\cU^\kappa:=\set{(a_1,a_2)\in \T\times \C:\abs{\im(a_2)}<\kappa (\abs{a_1}+\abs{\re(a_2)}),~\abs{\im(a_2)}<\pi/2}
\]
are subsets of $\cU$ with $\partial\cU^\kappa\cap\partial\cU=\set{0}$. Moreover, there exists a constant $C>0$ depending on $\kappa$ such that for all $a\in\cU^\kappa$, $j=0,1,2$ there holds 
\begin{align}\label{eq:boundK}
\abs{\partial_{a_2}^jK_2(a)}\leq  C\abs{a}_*^{-(1+j)}.
\end{align}
\end{lemma}
\begin{proof} Let $a\in \overline{\cU}^\kappa\cap \partial\cU$, $a_2=u+iv$. Then 
\begin{align*}
0=\cosh(a_2)-\cos(a_1)=\cosh(u)\cos(v)-\cos(a_1)+i\sinh(u)\sin(v)
\end{align*}
implies $v=0$ and thus $\cosh(u)=\cos(a_1)$, which is only possible for $a_1=u=0$; or $u=0$ and $\cos(v)=\cos(a_1)$, which in the closure of $\cU^\kappa$ is again only possible for $a_1=u=0$. Thus $\cU^\kappa\subset\cU$ and $\partial\cU^\kappa\cap \partial\cU=\set{0}$.

For the second part we split the analysis into three regions: $a\in\cU^\kappa$, $\abs{a}_*$ close to $0$; $a\in\cU^\kappa$, $\abs{a}_*$ large and the remaining subset of $\cU^\kappa$.

Let us start with $a\in\cU^\kappa$, $\abs{a}_*$ close to $0$. Writing again $a=(a_1,u+iv)$ and using that $v^2\leq 2\kappa^2(a_1^2+u^2)$ there holds
\begin{align*}
\abs{a_1^2+a_2^2}&=\left(a_1^4+2a_1^2(u^2-v^2)+(u^2-v^2)^2+4u^2v^2\right)^{\frac{1}{2}}\\
&\geq \left((1-4\kappa^2)a_1^4+2a_1^2(1-2\kappa^2)u^2+(u^2+v^2)^2\right)^{\frac{1}{2}}\\
&\geq (1-4\kappa^2)(a_1^4+\abs{a_2}^4)^{\frac{1}{2}}\geq \frac{1-4\kappa^2}{2}\abs{a}_*^2.
\end{align*}

Then for $a\in\cU^\kappa$, $a$ small it follows that 
\[
\abs{K_2(a)}\leq \frac{1}{2\pi}\frac{\abs{a_1}+O(\abs{a_1}^3)}{\abs{a_1^2+a_2^2}-O(\abs{a_1}^4)-O(\abs{a_2}^4)}\leq \frac{1}{\pi}\frac{\abs{a}_*+O(\abs{a}_*^3)}{(1-4\kappa^2)\abs{a}_*^2-O(\abs{a}_*^4)}\leq\frac{C}{\abs{a}_*}.
\]
Doing the same for higher order derivatives it follows that there exists $\varepsilon>0$ such that \eqref{eq:boundK} holds true for all $a\in\cU^\kappa$ with $\abs{a}_*<\varepsilon$. We fix such an $\varepsilon$.

Next let us consider the opposite regime $a=u+iv\in\cU^\kappa$, $\abs{a}_*$ large. Note that this necessarily means that $\abs{u}$ has to be large, since $\abs{a_1}\leq \pi$, $\abs{v}\leq \pi/2$ by definition of $\cU^\kappa$. We then estimate 
\begin{align*}
    \abs{\cosh(a_2)-\cos(a_1)}&\geq \abs{\cosh(u)\cos(v)+i\sinh(u)\sin(v)}-1\\
    &=\left(\cosh^2(u)-\sin^2(v)\right)^{\frac{1}{2}}-1\\
    &\geq \left(u^2/2-1\right)^{\frac{1}{2}}-1.
\end{align*}
Consequently one can find constants $C>0$ and $R>0$ such that \eqref{eq:boundK} with $j=0$ holds true for all $a\in\cU^\kappa$ with $\abs{a}_*>R$.  Again this procedure can be extended to higher order derivatives giving an $R$ as above but with \eqref{eq:boundK} valid for $j=0,1,2$ for $\abs{a}_*>R$. Let us also fix such an $R$.

Let now $a$ be in the remaining set, i.e. $a\in\cU^\kappa$ with $\varepsilon\leq \abs{a}_*\leq R$. The closure of this set is compact and bounded away from $\partial\cU$ where the denominator of $K_2$ vanishes. Therefore the existence of $C>0$ such that \eqref{eq:boundK} holds true also on this set just by continuity of $\partial^j_{a_2}K_2$, $j=0,1,2$. This finishes the proof of the lemma.
\end{proof}
\begin{lemma}\label{lem:initial_speed}
Let $\rho_0>0$ be chosen such that $\gamma_0$ extends holomorphically to $U_{2\rho_0}$ with
\begin{align}\label{eq:choice_of_rho0}
4\norm{\im(\gamma_0^\prime)}_{L^\infty(U_{\rho_0})}<1.
\end{align}
Then the complex extension of the initial normal velocity $s_0:\Omega_0\rightarrow\R$,
\[
s_0(y_1):=-2\int_{\T}K_2(\Delta X_0(y_1,z_1))\Delta\gamma_0'(y_1,z_1)\:dz_1
\]
is holomorphic on $U_{\rho_0}$. Moreover, the $L^\infty(U_{\rho_0})$-norm of any finite amount of derivatives of $s_0$ can be bounded by $C_0$. In particular, all derivatives of $s_0$ are given by differentiation under the integral.
\end{lemma}
\begin{proof}
By \eqref{eq:choice_of_rho0} one estimates
\[
\abs{\im(\gamma_0(y_1)-\gamma_0(y_1-z_1))}<\frac{1}{4}\abs{z_1}
\]
for $z_1\in[-\pi,\pi)$, $z_1\neq 0$, $y_1\in U_{\rho_0}$.
Thus by Lemma \ref{lem:good_sets_for_K_2} the composition of $K_2$ with $y_1\mapsto\Delta X_0(y_1,z_1)$ is holomorphic for every $z_1\neq 0$. Moreover, again by Lemma \ref{lem:good_sets_for_K_2} for such $z_1$ there holds
\[
\abs{\partial_{y_1}\big(K_2(\Delta X_0(y_1,z_1))\big)}\leq C\left(\frac{\abs{\Delta\gamma_0'(y_1,z_1)}^2}{\abs{\Delta X_0(y_1,z_1)}_*^2}+\frac{\abs{\Delta\gamma_0''(y_1,z_1)}}{\abs{\Delta X_0(y_1,z_1)}_*}\right)\leq C_0.
\] It follows that $s_0$ is holomorphic and that $\norm{s_0'}_{L^\infty(U_{\rho_0})}\leq C_0$. The same can be shown for higher order derivatives.
\end{proof}

The following two lemmas provide careful estimates needed for the compensation of various terms appearing in the definition of our nonlinear map $F$ below. We are also careful with the uniform integrability as we need to be able to neglect what happens in some small sets.
\begin{lemma}\label{lem:f_maps_into_good_set}
Let $\rho_0>0$ be as in Lemma \ref{lem:initial_speed} and let $R>0$. There exists $T=T(R,C_0,\alpha)\in(0,1)$ such that for all $\eta\in B_{\rho}$, $\norm{\eta}_{\rho}<R$, $\rho\in(0,\rho_0)$ and $t\in[0,T)$, $y\in \Omega_{\rho}$, $z\in\Omega_0$ there holds $\Delta X_t^\eta(y,z)\in\overline{\cU^{3/8}}$ and
\begin{gather}
t\abs{y_2-z_2}\leq C_0\abs{\Delta X^\eta_t(y,z)}_*.\label{eq:estimate_for_ty2-z2}
\end{gather}

\end{lemma}
\begin{proof}
First of all chose $T\in(0,1)$ with $T^\alpha R\leq 1$. Then, omiting the $(y,z)$ dependence in notation, cf. Sec. \ref{sec:notation}, there holds
\begin{align*}
\frac{t}{2}\abs{y_2-z_2}&\leq t\abs{y_2-z_2}-\frac{1}{2}t^{1+\alpha}R\abs{y_2-z_2}\\
&\leq \abs{t(y_2-z_2)+\frac{1}{2}t^{1+\alpha}\re(\eta(y)-\eta(y_1,z_2))}\\
&= \abs{\re\left(\Delta X^\eta_{t,2}-\Delta\gamma_0-t\Delta s_0-\frac{1}{2}t^{1+\alpha}(\eta(y_1,z_2)-\eta(y_1-z_1,z_2))\right)}\\
&\leq \abs{\re(\Delta X^\eta_{t,2})}+(C_0(1+T)+T^{1+\alpha}R)\abs{z_1}\leq\abs{\re(\Delta X^\eta_{t,2})}+C_0\abs{z_1}
\end{align*}
This implies \eqref{eq:estimate_for_ty2-z2}.

In order to see that $\Delta X^\eta_t\in \overline{\cU^{3/8}}$ we use \eqref{eq:choice_of_rho0}, as well as the just shown inequality to deduce
\begin{align*}
\abs{\im(\Delta X^\eta_{t,2})}&=\abs{\im\left( \Delta \gamma_0+t\Delta s_0+\frac{1}{2}t^{1+\alpha}\Delta\eta \right)}\\
&\leq \left(\frac{1}{4}+TC_0+T^{1+\alpha}R\right)\abs{z_1}+T^\alpha R t\abs{y_2-z_2}\\
&\leq \left(\frac{1}{4}+T(C_0+1)\right)\abs{z_1}+T^\alpha R\abs{\re(\Delta X^\eta_{t,2})}+T^\alpha R C_0\abs{z_1}.
\end{align*}
Thus by chosing $T>0$ even smaller the desired inequality
\begin{align*}
\abs{\im(\Delta X^\eta_{t,2})}\leq \frac{3}{8}\left(\abs{z_1}+\abs{\Delta\re(X_{t,2}^\eta)}\right)
\end{align*}
holds true.
\end{proof}

\begin{lemma}\label{lem:estimate_on_integral}
Let $\rho_0,R,T>0$ be as in Lemma \ref{lem:f_maps_into_good_set}. For $\eta\in B_{\rho}$, $\norm{\eta}_{\rho}<R$, $\rho\in(0,\rho_0)$ and $y\in\Omega_{\rho}$, $t\in(0,T)$ there holds
\begin{equation}
\int_{\Omega_0}\frac{1}{\abs{\Delta X^\eta_t(y,z)}_*}\:dz\leq C_0\abs{\log t}.
\end{equation}
The integrability of $\abs{\Delta X^\eta_t(y,\cdot)}_*^{-1}$ is uniform with respect to $y\in\Omega_0$ and with respect to $t$ considered on any interval of the form $[t_0,T)$ with $t_0>0$.
\end{lemma}
\begin{proof}
In view of \eqref{eq:estimate_for_ty2-z2} we have
\begin{align*}
\int_{\Omega_0}&\frac{1}{\abs{\Delta X^\eta_t}_*}\:dz\leq C_0\int_{\Omega_0}\frac{1}{\abs{z_1}+t\abs{y_2-z_2}}\:dz=C_0\int_{\T}\int_{y_2-2}^{y_2+2}\frac{1}{\abs{z_1}+t\abs{z_2}}\:dz_2\:dz_1\\
&\leq C_0\int_0^\pi\int_0^{4}\frac{1}{z_1+tz_2}\:dz_2\:dz_1=C_0\left(\frac{\pi}{t}\log\left(1+\frac{4t}{\pi}\right)+4\log\left(\frac{\pi}{4t}+1\right)\right),
\end{align*}
which is of order $\abs{\log t}$. Note here that $t<1$, since $T$ is assumed to be less than $1$.

The uniform integrability follows from $\frac{1}{\abs{z_1}+t_0\abs{z_2}}\in L^1(\T\times (-4,4))$.
\end{proof}

\subsection{\texorpdfstring{Definition of $F$}{Definition of F}}\label{sec:def_of_F}
Let us fix $\rho_0>0$ as in Lemma \ref{lem:initial_speed}. Take $R=1$ and a corresponding $T\in(0,1)$ from Lemma \ref{lem:f_maps_into_good_set}. 

We define the application $(\eta,t)\mapsto F(\eta,t)=F_t(\eta)$ by setting
\begin{equation*}\label{eq:definition_of_F}
F_t(\eta)(y):=-\int_{\Omega_0}K_2(\Delta X_t^\eta(y,z))\Delta\partial_{y_1}f_t^\eta(y,z)-K_2(\Delta X_0(y_1,z_1))\Delta\gamma_0'(y_1,z_1)\:dz
\end{equation*}
for $t>0$ and $F_0(\eta)(y)=0$.
\begin{lemma}\label{lem:F_welldef} $F$ when seen as a map
\begin{align*}
\set{\eta\in B_\rho:\norm{\eta}_{\rho}<1}\times [0,T)\rightarrow B_{\rho'}
\end{align*}
is well-defined for all $0<\rho'<\rho<\rho_0$. Moreover, for $\eta\in B_\rho$, $\norm{\eta}_{\rho}<1$ the map $[0,T)\rightarrow B_{\rho'}$, $t\mapsto F_t(\eta)$ is continuous.
\end{lemma}
\begin{proof} Let $\eta\in B_\rho$, $\norm{\eta}_\rho<1$ and $t\in(0,T)$. In view of Lemma \ref{lem:initial_speed} it remains to look at 
\begin{align*}
\tilde{F}_t(\eta)(y):=F_t(\eta)(y)+2s_0(y_1)=-\int_{\Omega_0}K_2(\Delta X_t^\eta(y,z))\Delta\partial_{y_1}f_t^\eta(y,z)\:dz.
\end{align*}

For $y\in\Omega_{\rho}$ and $z\in\Omega_0$ with $z_1\neq 0$ one computes
\begin{align}
\partial_{y_1}\big(K_2(\Delta X_t^\eta)\Delta\partial_{y_1}f^\eta_t \big)&=\partial_{a_2}K_2(\Delta X_t^\eta)(\Delta\partial_{y_1}f_t^\eta)^2+K_2(\Delta X^\eta_t)\Delta \partial_{y_1}^2f_t^\eta,\label{eq:formula_for_derivative_y1}\\
\begin{split}
\partial_{y_2}\big(K_2(\Delta X_t^\eta)\Delta\partial_{y_1}f^\eta_t \big)&=\partial_{a_2}K_2(\Delta X_t^\eta)\left(t+\frac{1}{2}t^{1+\alpha}\partial_{y_2}\eta(y)\right)\Delta\partial_{y_1}f_t^\eta\\
&\hspace{65pt}+K_2(\Delta X_t^\eta)\frac{1}{2}t^{1+\alpha}\partial_{y_2}\partial_{y_1}\eta(y),\label{eq:formula_for_derivative_y2}
\end{split}
\end{align}
where, as usual, we have omitted the $(y,z)$ dependence in the $\Delta$-notation.

In order to get uniform integrability we use \eqref{eq:estimate_for_ty2-z2} to estimate
\begin{align}
\begin{split}\label{eq:estimate_for_Delta_partialy1_f}
\abs{\Delta \partial^j_{y_1}f_t^\eta}&\leq C_0\abs{z_1}+\norm{\partial_{y_1}^{j+1}\eta}_{L^\infty(\Omega_{\rho'})}\abs{z_1}+\norm{\partial_{y_2}\partial_{y_1}^j\eta}_{L^\infty(\Omega_{\rho'})}t\abs{y_2-z_2}\\
&\leq C_0\left(1+\norm{\partial^j_{y_1}\nabla_{y}\eta}_{L^\infty(\Omega_{\rho'})}\right)\abs{\Delta X_t^\eta}_*
\end{split}
\end{align}
for $y\in\Omega_{\rho'}$, $j=0,1,2$. Now \eqref{eq:estimate_for_Delta_partialy1_f} and Lemmas \ref{lem:good_sets_for_K_2}, \ref{lem:f_maps_into_good_set} imply
\begin{align}\label{eq:majorant1_function}
\abs{K_2(\Delta X_t^\eta)\Delta\partial_{y_1}f^\eta_t}\leq C_0\left(1+\norm{\partial_{y_1}\nabla_{y}\eta}_{L^\infty(\Omega_{\rho'})}\right).
\end{align}

As a consequence $\tilde{F}_t$, and thus $F_t$, maps at least into $L^\infty(\Omega_{\rho'})$. Moreover, combining similarly \eqref{eq:boundK}, \eqref{eq:estimate_for_Delta_partialy1_f} for $j=1,2$, and \eqref{eq:estimate_for_ty2-z2} to estimate 
\eqref{eq:formula_for_derivative_y1}, \eqref{eq:formula_for_derivative_y2} one sees that
\begin{align}\label{eq:majorant2_partialy1}
\abs{\partial_{y_1}\big(K_2(\Delta X_t^\eta)\Delta\partial_{y_1}f^\eta_t \big)}&\leq C_0\left(1+\norm{\partial_{y_1}\nabla_y\eta}_{L^\infty(\Omega_{\rho'})}^2+\norm{\partial_{y_1}^2\nabla_y\eta}_{L^\infty(\Omega_{\rho'})}\right),
\end{align}
and, recall $t<1$, $\abs{\partial_{y_2}\eta(y)}<1$, that
\begin{align}\label{eq:majorant3_partialy2}
\abs{\partial_{y_2}\big(K_2(\Delta X_t^\eta)\Delta\partial_{y_1}f^\eta_t \big)}&\leq C_0\frac{t}{\abs{\Delta X_t^\eta}_*}\left(1+\norm{\partial_{y_1}\nabla_y\eta}_{L^\infty(\Omega_{\rho'})}\right).
\end{align}
It follows that the complex derivative $\partial_{y_1}F_t(\eta)$ exists and is bounded on $\Omega_{\rho'}$. Moreover, in view of Lemma \ref{lem:estimate_on_integral} the same is true for the (real) derivative $\partial_{y_2}F_t(\eta)$.

Next we turn to the required uniform continuity of $\partial_{y_2}F_t(\eta)$ on $\Omega_{\rho'}$. First of all observe that the corresponding integrant \eqref{eq:formula_for_derivative_y2} as a function of $(z,y)\in\Omega_0\times\Omega_{\rho'}$ is uniformly continuous on subsets which have their $z_1$-component bounded away from $0$. Here one uses the Cauchy integral formula and the assumed uniform continuity of $\partial_{y_2}\eta$ on the larger set $\Omega_\rho$ in order to conclude the uniform continuity of $\partial_{y_2}\partial_{y_1}\eta(y)=\partial_{y_1}\partial_{y_2}\eta(y)$. This together with the uniform integrability of the majorant given in \eqref{eq:majorant3_partialy2} via Lemma \ref{lem:estimate_on_integral} implies that $\partial_{y_2}F_t(\eta)$ is uniformly continuous on $\Omega_{\rho'}$, cf. also the argument below for continuity in time.

Moreover, in a similar way as above for \eqref{eq:majorant1_function}, \eqref{eq:majorant2_partialy1}, \eqref{eq:majorant3_partialy2} one can check that for any $y\in\Omega_{\rho'}$, $z_1\neq 0$ there holds
\begin{align*}
\abs{\partial_{y_1}\partial_{y_2}\big(K_2(\Delta X_t^\eta)\Delta\partial_{y_1}f^\eta_t \big)}\leq \frac{C_0t}{\abs{\Delta X_t^\eta}_*}\left(1+\norm{\partial_{y_1}\nabla_y\eta}_{L^\infty(\Omega_{\rho'})}^2+\norm{\partial_{y_1}^2\nabla_y\eta}_{L^\infty(\Omega_{\rho'})}\right),
\end{align*}
which implies that also $\partial_{y_2}F_t(\eta)$ is complex differentiable in $y_1$.

In order to conclude $F_t(\eta)\in B_{\rho'}$ it therefore only remains to observe that $F_t(y)\in\R$ for $y\in\Omega_0$.

It remains to prove the continuity of $[0,T)\ni t\mapsto F_t(\eta)\in B_{\rho'}$. Let $t,t_0\in(0,T)$ and take $\delta>0$ sufficiently small. For $z\in\Omega_0$ with $\abs{z_1}>\delta$, as well as $y\in \Omega_{\rho'}$ there holds
\begin{align*}
\abs{K_2(\Delta X_t^\eta)\Delta \partial_{y_1}f_t^\eta-K_2(\Delta X_{t_0}^\eta)\Delta \partial_{y_1}f_{t_0}^\eta}\leq \frac{C_0}{\delta^2}\abs{t-t_0}
\end{align*}
due to Lemmas \ref{lem:good_sets_for_K_2}, \ref{lem:f_maps_into_good_set}. On the set $\set{z\in\Omega_0:\abs{z_1}<\delta}$ one uses the uniform majorant given in \eqref{eq:majorant1_function} to conclude the continuity of $(0,T)\ni t\mapsto F_t(\eta)$ with respect to $\norm{\cdot}_{L^\infty(\Omega_{\rho'})}$. 

For the corresponding continuity of $\partial_{y_1}F_t(\eta)$, $\partial_{y_2}F_t(\eta)$ with respect to $\norm{\cdot}_{L^\infty(\Omega_{\rho'})}$ one uses a similar combination of Lipschitz continuity on $\abs{z_1}>\delta$ and uniform integrability on the strip $\abs{z_1}<\delta$ induced by \eqref{eq:majorant2_partialy1}, \eqref{eq:majorant3_partialy2} and Lemma \ref{lem:estimate_on_integral}. 

Finally, continuity at $t_0=0$ can be shown in the exact same way by noting that compared to Lemma \ref{lem:estimate_on_integral} the additional factor $t$ in \eqref{eq:majorant3_partialy2} for $\partial_{y_2}F_t(\eta)$ renders $t\abs{\Delta X_t^\eta}_*^{-1}$ to be uniformly integrable with respect to $t$ taken from the open interval $t\in(0,T)$.
\end{proof}

\begin{remark}
The continuity of $F$ as stated in assumption \ref{cond:F_well_defined} of Theorem \ref{thm:abstract_CK} will follow from Lemma \ref{lem:F_welldef} when combined with the Lipschitz property of Lemma \ref{lem:contraction} below.
\end{remark}

\subsection{Contraction property}\label{sec:contraction}
Next we will verify assumption \ref{cond:F_lipschitz} of Theorem \ref{thm:abstract_CK} with $b(t)=C_0t^{1+\alpha}\abs{\log t}$.
Let $\rho_0,R,T>0$ be as in Section \ref{sec:def_of_F}. Recall that $R=1$ and $T=T(R,C_0,\alpha)<1$. Without loss of generality we also assume $\rho_0<1$.
\begin{lemma}\label{lem:contraction}
For all $0<\rho'<\rho<\rho_0$, $\eta,\zeta\in B_\rho$, $\norm{\eta}_\rho<1$, $\norm{\zeta}_\rho<1$ and $t\in[0,T)$ there holds
\begin{align*}
\norm{F_t(\eta)-F_t(\zeta)}_{\rho'}\leq \frac{C_0t^{1+\alpha}\abs{\log t}}{\rho-\rho'}\norm{\eta-\zeta}_\rho.
\end{align*}
\end{lemma}

For the proof of Lemma \ref{lem:contraction} we first of all state some estimates implied by the Lemmas in Section \ref{sec:preliminary_lemmas}. 
\begin{lemma}\label{lem:contraction_auxiliary}
Let $0<\rho'<\rho<\rho_0$ and $\eta,\xi,\zeta\in B_\rho$ with $\norm{\eta}_\rho,\norm{\xi}_\rho,\norm{\zeta}_\rho<1$. For $y\in\Omega_{\rho'}$, $z\in\Omega_0$, $t\in[0,T)$ there holds
\begin{gather}
t\abs{\Delta \zeta(y,z)}\leq C_0\norm{\zeta}_{\rho'}|\Delta X_t^\xi(y,z)|_*,\label{eq:contraction_aux4}\\
t\abs{\Delta \partial_{y_1}\zeta(y,z)}\leq \frac{C_0}{\rho-\rho'}\norm{\zeta}_\rho|\Delta X_t^\xi(y,z)|_*,\label{eq:contraction_aux5}\\
\abs{\Delta \partial_{y_1}f_t^\eta(y,z)}\leq \frac{C_0}{\rho-\rho'}|\Delta X_t^\xi(y,z)|_*,\label{eq:contraction_aux1}\\
\abs{\Delta \zeta(y,z)\Delta\partial_{y_1}f_t^\eta(y,z)}\leq C_0\norm{\zeta}_{\rho'}|\Delta X_t^\xi(y,z)|_*,\label{eq:contraction_aux2}\\
\abs{\Delta\zeta(y,z)\Delta\partial_{y_1}^2f_t^\eta(y,z)}\leq \frac{C_0}{\rho-\rho'}\norm{\zeta}_{\rho'}|\Delta X_t^\xi(y,z)|_*.\label{eq:contraction_aux3}
\end{gather}
\end{lemma}
\begin{proof}
By \eqref{eq:estimate_for_ty2-z2} in Lemma \ref{lem:f_maps_into_good_set} one deduces
\begin{align*}
t\abs{\Delta \zeta}\leq \norm{\nabla_y\zeta}_{L^\infty(\Omega_{\rho'})}\left(\abs{z_1}+t\abs{y_2-z_2}\right)\leq C_0\norm{\zeta}_{\rho'}|\Delta X_t^\xi(y,z)|_*.
\end{align*}
This shows \eqref{eq:contraction_aux4}. Inequality \eqref{eq:contraction_aux5} is obtained in the same way by additionally applying Cauchy's Lemma \ref{lem:cauchy}.

Next, \eqref{eq:contraction_aux1} follows from
\begin{align*}
\abs{\Delta\partial_{y_1}f_t^\eta}&\leq C_0\abs{z_1}+t^{1+\alpha}\abs{\Delta\partial_{y_1}\eta}
\end{align*}
and \eqref{eq:contraction_aux5}, while \eqref{eq:contraction_aux2} is a consequence of 
\begin{align*}
\abs{\Delta\zeta\Delta\partial_{y_1}f_t^\eta}\leq C_0\norm{\zeta}_{L^\infty(\Omega_{\rho'})}\abs{z_1}+t^{1+\alpha}\abs{\Delta\zeta}\norm{\partial_{y_1}\eta}_{L^\infty(\Omega_{\rho'})}
\end{align*}
and \eqref{eq:contraction_aux4}.

Finally, \eqref{eq:contraction_aux3} is achieved in the same way as \eqref{eq:contraction_aux2}, but with an additional use of Lemma \ref{lem:cauchy}.
\end{proof}

\begin{proof}[Proof of Lemma \ref{lem:contraction}] Let $0<\rho'<\rho<\rho_0$ and $\eta,\zeta\in B_\rho$ be as stated. For $\lambda\in[0,1]$ define 
\[
\xi_\lambda:=\lambda\eta+(1-\lambda)\zeta.
\]
Then $\xi_\lambda\in B_\rho$ and $\norm{\xi_\lambda}_{\rho}<1$. 

Now for $y\in\Omega_{\rho'}$ we split
\begin{align}\begin{split}\label{eq:split_of_F_t}
\abs{F_t(\eta)(y)-F_t(\zeta)(y)}&\leq\int_{\Omega_0}\abs{\left(K_2(\Delta X_t^\eta)-K_2(\Delta X_t^\zeta)\right)\Delta \partial_{y_1}f_t^\eta}\:dz\\
&\hspace{25pt}+\int_{\Omega_0}\abs{K_2(\Delta X_t^\zeta)\left(\Delta \partial_{y_1}f_t^\eta-\Delta\partial_{y_1}f_t^\zeta\right)}\:dz.
\end{split}
\end{align}

In order to estimate the first term we first use the fundamental theorem of Calculus to write

\begin{align*}
\int_{\Omega_0}&|K_2(\Delta X_t^\eta)-K_2(\Delta X_t^\zeta)|\abs{\Delta \partial_{y_1}f_t^\eta}\:dz\\
&=\int_{\Omega_0}\abs{\int_0^1 \partial_{a_2}K_2(\Delta X^{\xi_\lambda}_t)\frac{1}{2}t^{1+\alpha}(\Delta \eta-\Delta \zeta)\:d\lambda}\abs{\Delta \partial_{y_1}f_t^\eta}\:dz.
\end{align*}
Now, $\abs{\Delta \partial_{y_1}f_t^\eta}$
is dealt with by  \eqref{eq:contraction_aux1} with $\xi=\xi_\lambda$, Lemma \ref{lem:good_sets_for_K_2} and its equation \eqref{eq:boundK} deal with $\partial_{a_2}K_2(\Delta X^{\xi_\lambda}_t)$ and by definition of $\|\cdot\|_{\rho'}$ we arrive to the estimate
\begin{align*}
\int_{\Omega_0}&|K_2(\Delta X_t^\eta)-K_2(\Delta X_t^\zeta)|\abs{\Delta \partial_{y_1}f_t^\eta}\:dz\\
&=\int_{\Omega_0}\abs{\int_0^1 \partial_{a_2}K_2(\Delta X^{\xi_\lambda}_t)\frac{1}{2}t^{1+\alpha}(\Delta \eta-\Delta \zeta)\:d\lambda}\abs{\Delta \partial_{y_1}f_t^\eta}\:dz\\
&\leq C_0t^{1+\alpha}\norm{\eta-\zeta}_{\rho'}\int_{\Omega_0}\int_0^1\frac{1}{|\Delta X_t^{\xi_\lambda}|_*^2}\frac{|\Delta X_t^{\xi_\lambda}|_*}{\rho-\rho'}\:d\lambda\:dz\leq \frac{C_0t^{1+\alpha}\abs{\log t}}{\rho-\rho'}\norm{\eta-\zeta}_{\rho},
\end{align*}
where the last inequality is a direct application of Lemma \ref{lem:estimate_on_integral}.
Again by Lemmas \ref{lem:good_sets_for_K_2}, \ref{lem:estimate_on_integral} the second term in \eqref{eq:split_of_F_t} is bounded by
\begin{align*}
\int_{\Omega_0}\abs{K_2(\Delta X_t^\zeta)\left(\Delta \partial_{y_1}f_t^\eta-\Delta\partial_{y_1}f_t^\zeta\right)}\:dz&\leq C\int_{\Omega_0}\frac{1}{|\Delta X_t^\zeta|_*}t^{1+\alpha}\abs{\Delta \partial_{y_1}\eta-\Delta\partial_{y_1}\zeta}\:dz\\
&\leq C_0t^{1+\alpha}\abs{\log t}\norm{\eta-\zeta}_{\rho}.
\end{align*}
Thus 
\begin{align*}
\norm{F_t(\eta)-F_t(\zeta)}_{L^\infty(\Omega_{\rho'})}\leq \frac{C_0t^{1+\alpha}\abs{\log t}}{\rho-\rho'}\norm{\eta-\zeta}_{\rho}.
\end{align*}

Let us now turn to the corresponding inequality with $\partial_{y_1}$. In a similar way as before we split
\begin{align*}
\partial_{y_1}F_t(\eta)(y)-\partial_{y_1}F_t(\zeta)(y)=-\int_{\Omega_0} A_1+A_2+A_3+A_4\:dz,
\end{align*}
where 
\begin{align*}
A_1&:=\left(\partial_{a_2}K_2(\Delta X_t^\eta)-\partial_{a_2}K_2(\Delta X_t^\zeta)\right)(\Delta \partial_{y_1}f_t^\eta)^2,\\
A_2&:=\partial_{a_2}K_2(\Delta X_t^\zeta)\left((\Delta \partial_{y_1}f_t^\eta)^2-(\Delta \partial_{y_1}f_t^\zeta)^2\right),\\
A_3&:=\left(K_2(\Delta X_t^\eta)-K_2(\Delta X_t^\zeta)\right)\Delta \partial^2_{y_1}f_t^\eta,\\
A_4&:=K_2(\Delta X_t^\zeta)\left(\Delta \partial^2_{y_1}f_t^\eta-\Delta \partial^2_{y_1}f_t^\zeta\right),
\end{align*}
cf. \eqref{eq:formula_for_derivative_y1}. Regarding $A_1$ we use \eqref{eq:contraction_aux1}, \eqref{eq:contraction_aux2} to deduce
\begin{align*}
\int_{\Omega_0}\abs{A_1}\:dz&\leq C\int_{\Omega_0}\int_0^1\frac{1}{|\Delta X_t^{\xi_\lambda}|_*^3}t^{1+\alpha}\abs{\Delta(\eta-\zeta)\Delta \partial_{y_1}f_t^\eta}\abs{ \Delta \partial_{y_1}f_t^\eta }\:dz\:d\lambda\\
&\leq \frac{C_0t^{1+\alpha}}{\rho-\rho'}\norm{\eta-\zeta}_{\rho'}\int_0^1\int_{\Omega_0}\frac{1}{|\Delta X_t^{\xi_\lambda}|_*}\:dz\:d\lambda\leq \frac{C_0t^{1+\alpha}\abs{\log t}}{\rho-\rho'}\norm{\eta-\zeta}_{\rho}.
\end{align*}
By making use of \eqref{eq:contraction_aux3} instead of \eqref{eq:contraction_aux2} one can bound $\int_{\Omega_0}\abs{A_3}\:dz$ in a similar way. We omit the details.

Next for $A_2$, inequality \eqref{eq:contraction_aux1} implies 
\begin{align*}
\int_{\Omega_0}\abs{A_2}\:dz&\leq C\int_{\Omega_0}\frac{1}{|\Delta X_t^{\zeta}|^2_*}\abs{\Delta \partial_{y_1}f_t^\eta+\Delta \partial_{y_1}f_t^\zeta}t^{1+\alpha}\abs{\Delta\partial_{y_1}\eta-\Delta\partial_{y_1}\zeta}\:dz\\
&\leq \frac{C_0t^{1+\alpha}\abs{\log t}}{\rho-\rho'}\norm{\eta-\zeta}_{\rho}.
\end{align*}
Finally, the estimate for $\int_{\Omega_0}\abs{A_4}\:dz$ is a straightforward consequence of Lemmas \ref{lem:good_sets_for_K_2}, \ref{lem:f_maps_into_good_set}, \ref{lem:estimate_on_integral} and the Cauchy Lemma \ref{lem:cauchy}.

Summarizing, we have shown
\begin{align*}
\norm{\partial_{y_1}F_t(\eta)-\partial_{y_1}F_t(\zeta)}_{L^\infty(\Omega_{\rho'})}\leq \frac{C_0t^{1+\alpha}\abs{\log t}}{\rho-\rho'}\norm{\eta-\zeta}_{\rho}.
\end{align*}

It therefore remains to check $\partial_{y_2}$. Again we write
\begin{align*}
\partial_{y_2}F_t(\eta)(y)-\partial_{y_2}F_t(\zeta)(y)=-\int_{\Omega_0}B_1+B_2+B_3+B_4\:dz,
\end{align*}
where, cf. \eqref{eq:formula_for_derivative_y2},
\begin{align*}
B_1&:=\left(\partial_{a_2}K_2(\Delta X_t^\eta)-\partial_{a_2}K_2(\Delta X_t^\zeta)\right)\left(t+\frac{1}{2}t^{1+\alpha}\partial_{y_2}\eta(y)\right)\Delta \partial_{y_1}f_t^\eta,\\
B_2&:=\partial_{a_2}K_2(\Delta X_t^\zeta)\left[\left(t+\frac{1}{2}t^{1+\alpha}\partial_{y_2}\eta(y)\right)\Delta \partial_{y_1}f_t^\eta-\left(t+\frac{1}{2}t^{1+\alpha}\partial_{y_2}\zeta(y)\right)\Delta \partial_{y_1}f_t^\zeta\right],\\
B_3&:=\left(K_2(\Delta X_t^\eta)-K_2(\Delta X_t^\zeta)\right)\frac{1}{2}t^{1+\alpha}\partial_{y_2}\partial_{y_1}\eta(y),\\
B_4&:=K_2(\Delta X_t^\zeta)\frac{1}{2}t^{1+\alpha}\left(\partial_{y_2}\partial_{y_1}\eta(y)-\partial_{y_2}\partial_{y_1}\zeta(y)\right).
\end{align*}

Since $t<1$ and $\abs{\partial_{y_2}\eta(y)}<1$ we get
\begin{align*}
\int_{\Omega_0}\abs{B_1}\:dz&\leq Ct^{1+\alpha}\int_0^1\int_{\Omega_0}\frac{1}{|\Delta X_t^{\xi_\lambda}|^3_*}t\abs{\Delta\eta-\Delta \zeta}\abs{\Delta\partial_{y_1}f_t^\eta}\:dz\:d\lambda\\
&\leq \frac{C_0t^{1+\alpha}\abs{\log t}}{\rho-\rho'}\norm{\eta-\zeta}_\rho
\end{align*}
by \eqref{eq:contraction_aux4}, \eqref{eq:contraction_aux1}.

Moreover, 
\begin{align*}
\int_{\Omega_0}\abs{B_2}\:dz&\leq C\int_{\Omega_0}\frac{1}{|\Delta X_t^\zeta|_*^2}\Big[t^{1+\alpha}\abs{\partial_{y_2}\eta(y)-\partial_{y_2}\zeta(y)}\abs{\Delta\partial_{y_1}f_t^\eta}\\
&\hspace{65pt}+\left(t+t^{1+\alpha}\abs{\partial_{y_2}\zeta(y)}\right)t^{1+\alpha}\abs{\Delta\partial_{y_1}\eta-\Delta\partial_{y_1}\zeta}\Big]\:dz\\
&\leq \frac{C_0t^{1+\alpha}\abs{\log t}}{\rho-\rho'}\norm{\eta-\zeta}_\rho
\end{align*}
by use of \eqref{eq:contraction_aux1} in the first term, as well as \eqref{eq:contraction_aux5} in the second.

The estimate for $\int_{\Omega_0}\abs{B_3}\:dz$ follows in analogy to $\int_{\Omega_0}\abs{B_1}\:dz$ by utilizing \eqref{eq:contraction_aux4} and Cauchy's Lemma \ref{lem:cauchy}, whereas the estimate for $\int_{\Omega_0}\abs{B_4}\:dz$ relies solely on Lemma \ref{lem:cauchy}.

This finishes the proof of Lemma \ref{lem:contraction}.
\end{proof}

\subsection{The affine term}\label{sec:affine_term}
In order to complete the list of ingredients of Theorem \ref{thm:abstract_CK}, we investigate $F_t(0)$. As usual we consider $\rho_0\in(0,1)$ to be fixed according to Lemma \ref{lem:initial_speed}, as well as $R=1$ and $T=T(R,C_0,\alpha)\in(0,1)$ given by Lemma \ref{lem:f_maps_into_good_set}.
\begin{lemma}\label{lem:estimate_for_F0}
For any $\rho\in(0,\rho_0)$, $t\in(0,T)$ there holds
\begin{align*}
\norm{F_t(0)}_\rho\leq C_0 t\abs{\log t}.
\end{align*}
\end{lemma}
\begin{proof}
Let $y\in \Omega_\rho$. Recall that 
\begin{align*}
\Delta X_t^0=\begin{pmatrix}
z_1\\
\Delta\gamma_0+t\Delta s_0+t(y_2-z_2)
\end{pmatrix}
=\Delta X_0+\begin{pmatrix}
0\\t\Delta s_0+t(y_2-z_2)
\end{pmatrix},
\end{align*}
 $z\in\Omega_0$. In view of Lemmas \ref{lem:good_sets_for_K_2}, \ref{lem:f_maps_into_good_set}, \ref{lem:estimate_on_integral} and the boundedness of $\Omega_0$ there holds
\begin{align*}
\abs{F_t(0)(y)}&\leq \int_{\Omega_0}\abs{K_2(\Delta X_t^0)-K_2(\Delta X_0)}\abs{\Delta\gamma_0'}+t\abs{K_2(\Delta X_t^0)}\abs{\Delta s_0'}\:dz\\
&\leq \int_0^1\int_{\Omega_0}\abs{\partial_{a_2}K_2(\Delta X_{\lambda t}^0)t(y_2-z_2+\Delta s_0)}\abs{\Delta \gamma_0'}\:dz \:d\lambda+C_0t\\
&\leq C_0t\left(1+\int_0^1\int_{\Omega_0}\frac{1}{\abs{\Delta X_{\lambda t}^0}_*}\:dz\:d\lambda\right)\leq C_0t\left(1+\int_0^1\abs{\log(\lambda t)}\:d\lambda\right)\\
&\leq C_0 t\abs{\log t}.
\end{align*}

Next we estimate $\partial_{y_1}F_t(0)$. One has
\begin{align*}
&\abs{\partial_{y_1}F_t(0)(y)}\leq \int_{\Omega_0}\abs{\partial_{a_2}K_2(\Delta X_t^0)-\partial_{a_2}K_2(\Delta X_0)}\abs{\Delta \gamma_0'}^2+\abs{K_2(\Delta X_t^0)}t\abs{\Delta s_0''}\\
&\hspace{25pt}+\abs{K_2(\Delta X_t^0)-K_2(\Delta X_0)}\abs{\Delta \gamma_0''}+\abs{\partial_{a_2}K_2(\Delta X_t^0)}(2t\Delta \gamma_0'\Delta s_0'+t^2\abs{\Delta s_0'}^2)\:dz.
\end{align*}
The terms appearing on the right-hand side can be dealt with in a similar way as above.

Finally, we also state
\begin{align*}
\abs{\partial_{y_2}F_t(0)(y)}&\leq \int_{\Omega_0}\abs{\partial_{a_2}K_2(\Delta X_t^0)}t\abs{\Delta\gamma_0'+t\Delta s_0'}\:dz\leq C_0t\abs{\log t}.
\end{align*}
This finishes the proof of Lemma \ref{lem:estimate_for_F0}.
\end{proof}

\begin{remark}
Note that Lemma \ref{lem:estimate_for_F0} implies
\begin{align*}
\frac{1}{t^{1+\alpha}}\int_0^t \norm{F_s(0)}_\rho\:ds\leq C_0t^{1-\alpha}\abs{\log t}\leq \frac{C_0}{\rho_0-\rho}t^{1-\alpha}\abs{\log t},
\end{align*}
i.e. assumption \ref{cond:F_0_well_defined} of Theorem \ref{thm:abstract_CK} holds with $a(t)=t^{1+\alpha}$, $c(t)=C_0t^{1-\alpha}\abs{\log t}$.
\end{remark}

\subsection{Conclusion and additional remarks}\label{sec:conclusion_of_existence}

In Sections \ref{sec:banach_spaces}-\ref{sec:affine_term} we have verified all conditions of Theorem \ref{thm:abstract_CK}. As a consequence we deduce the following statement.
\begin{proposition}\label{prop:conclusion1} Let $\rho_0>0$ be as in Lemma \ref{lem:initial_speed}. There exists $\bar{a}=\bar{a}(C_0)>0$, $T=T(C_0,\alpha)>0$ and a unique function $t\mapsto \eta_t$ with the properties that for every $\rho\in(0,\rho_0)$ the map 
    \[
    I_\rho:=\set{t\in[0,T):C_0t^{1-\alpha}\abs{\log t}<\bar{a}(\rho_0-\rho)}\ni t\mapsto\eta_t\in B_\rho
    \]
    is continuous with $\norm{\eta_t}_\rho<1$, $t\in I_\rho$, and such that 
    for all $y\in\Omega_{\rho}$, $t\in I_\rho$ there holds \begin{equation}\label{eq:equation_for_eta_in_conclusion}
\eta_0(y)=0,\quad \eta_t(y)=\frac{1}{t^{1+\alpha}}\int_0^tF_s(\eta_s)(y)\:ds.
\end{equation}
\end{proposition}

We finish the investigation of the fixed-point problem \eqref{eq:fixed_point_equation_eta} with some accompanying remarks concerning properties of the solution $\eta_t$ given by Proposition \ref{prop:conclusion1}.

The first addresses regularity. In contrast to the analyticity of $\eta_t$ in $y_1$, we only know that $\eta_t$ is continuously differentiable in $y_2$. Using \eqref{eq:equation_for_eta_in_conclusion} it seems possible to upgrade the regularity with respect to $y_2$. However, since $F_s(\eta_s)(y)$ involves the integration over the finite interval $(-2,2)$ with respect to $z_2$, in contrast to $\T$ for the integration in $z_1$, the maximal regularity for $\eta_t(y_1,\cdot):[-2,2]\rightarrow\C$ is expected to be finite. In any case, since a higher regularity of $\eta_t$ with respect to $y_2$ would only improve the regularity of our subsolution inside the mixing zone and not across its boundary, we have not pursued this topic any further.

Next we turn to the role of the parameter $\alpha$. Suppose that we set up problem \eqref{eq:fixed_point_equation_eta} with respect to two different choices $0<\alpha<\beta<1$ leading to two different right-handsides involving $F_t^\alpha(\eta)$, $F_t^\beta(\eta)$. Our previous analysis gives two solutions $\eta_t^\alpha$, $\eta_t^\beta$ with corresponding intervals $I_\rho^\alpha\subset [0,T^\alpha)$, $I_\rho^\beta\subset[0,T^\beta)$, $\rho\in(0,\rho_0)$. Note that the intervals $I^\alpha_\rho$, $I^\beta_\rho$ are defined with the same $\bar{a}$ and recall that $T^\alpha,T^\beta\in(0,1)$.
\begin{lemma}\label{lem:independence_of_alpha} 
There holds $t^{\beta-\alpha}\eta_t^\beta=\eta^\alpha_t$ on $[0,\min\set{T^\alpha,T^\beta})$.
\end{lemma}
\begin{proof} Define $J_\rho^\alpha:=I_\rho^\alpha\cap[0,T^\beta)$, $J_\rho^\beta:=I_\rho^\beta\cap[0,T^\alpha)$. Then $J^\beta_\rho\subset J^\alpha_\rho$ due to the fact that $\beta>\alpha$ and $t<1$. Both functions $t^{\beta-\alpha}\eta_t^\beta$, $\eta_t^\alpha$ are continuous maps from $J^\beta_\rho$ into the unit ball of $B_\rho$, $\rho\in(0,\rho_0)$, and they both vanish at $t=0$. Moreover, it is easy to check that 
\begin{align*}
    t^{\beta-\alpha}\eta_t^\beta(y)=t^{\beta-\alpha}\frac{1}{t^{1+\beta}}\int_0^tF_s^\beta(\eta_s^\beta)(y)\:ds=\frac{1}{t^{1+\alpha}}\int_0^tF_s^\alpha(s^{\beta-\alpha}\eta^\beta_s)(y)\:ds.
    \end{align*}
Thus Proposition \ref{prop:conclusion1} implies $t^{\beta-\alpha}\eta_t^\beta=\eta_t^\alpha$ for as long as both are defined.
\end{proof}
Both solutions $t^{\beta-\alpha}\eta^\beta_t$, $\eta_t^\alpha$ of \eqref{eq:equation_for_eta_in_conclusion} then extend uniquely to a common maximal solution of \eqref{eq:equation_for_eta_in_conclusion} enjoying the properties of Proposition \ref{prop:conclusion1}. Moreover, Lemma \ref{lem:independence_of_alpha} shows that $t^{1+\alpha}\eta^\alpha_t$ is independent of the considered $\alpha\in(0,1)$, such that the induced function $f(t,y)$ defined in Section \ref{sec:justification} below, is independent of $\alpha\in(0,1)$.

Finally we remark that for the choice $\alpha=1$ in ansatz \eqref{eq:ansatz_for_f} a more careful analysis would have been required. In that case the initial value $\eta_0(y)$ is not expected to be given by $0$ and the estimate given in Lemma \ref{lem:estimate_for_F0} does not even lead to boundedness of $t^{-2}\int_0^t\norm{F_s(0)}_\rho\:ds$. However, since this analysis has not been needed in order to prove existence of a Lipschitz solution of \eqref{eq:subequation}, we leave the case $\alpha=1$ as a possible future improvement.

\section{Justification of ansatzes}\label{sec:justification}

We will now verify that $\eta$ provided by Proposition \ref{prop:conclusion1} induces when undoing the transformations stated in Section \ref{sec:level_set_formulation} indeed a solution of the macroscopic IPM system \eqref{eq:subequation}.

Given $\eta$ from Proposition \ref{prop:conclusion1} we first of all define  $f:[0,T)\times\overline{\Omega}_{0}\rightarrow \R$, 
\[
f(t,y):=f_t^{\eta_t}(y)=\gamma_0(y_1)+ts_0(y_1)+\frac{1}{2}t^{1+\alpha}\eta_t(y),
\]
where $T=T(C_0,\alpha)>0$ can be taken as the endpoint of the interval $I_{\rho_0/2}$ for instance. Also recall Lemma \ref{lem:real_C1} if needed for the extension to the closure of $\Omega_0$.
\begin{lemma}\label{lem:back_to_the_real_case} There holds $f\in\cC^1([0,T);\cC^{1}(\overline{\Omega}_0))$  with
\begin{equation}\label{eq:bound_on_partialy2_f}
    \norm{\partial_{y_2}f(t,\cdot)}_{L^\infty(\Omega_0)}\leq \frac{1}{2}t^{1+\alpha},~t\in(0,T).
\end{equation}
Moreover, the functions $f(t,\cdot,y_2)$, $\partial_tf(t,\cdot,y_2)$, $\partial_{y_2}f(t,\cdot,y_2)$, $t\in[0,T)$, $y_2\in[-2,2]$ are real analytic and $f$ satisfies the initial value problem $f(0,y)=\gamma_0(y_1)$,
\begin{align}\label{eq:equation_for_f_after_conclusion}
\partial_tf(t,y)=-\frac{1}{2}\int_{\Omega_0}K_2(\Delta X_t^{\eta_t}(y,z))(\partial_{y_1}f(t,y)-\partial_{y_1}f(t,y_1-z_1,z_2))\:dz
\end{align}
for $t\in[0,T)$, $y\in\overline{\Omega}_0$.
\end{lemma}
\begin{proof}
As a direct consequence of Lemma \ref{lem:F_welldef} and Proposition \ref{prop:conclusion1} one sees that $f$ belongs to $\cC^1([0,T);B_{\rho_0/2})$ and satisfies \eqref{eq:bound_on_partialy2_f}, \eqref{eq:equation_for_f_after_conclusion} for $t\in[0,T)$, $y\in\Omega_{\rho_0/2}$. The statement follows from the definition of the spaces $B_\rho$ and Lemma \ref{lem:real_C1}.
\end{proof}

We are now able to prove our main result.
\begin{proof}[Proof of Theorem \ref{thm:main}] Let $f$ be as in Lemma \ref{lem:back_to_the_real_case}. Define the open space-time set 
\[
\mathscr{U}:=\set{(t,x)\in(0,T)\times\T\times\R:-2t+f(t,x_1,-2)<x_2<2t+f(t,x_1,2)},
\]
as well as the slices 
\[
\mathscr{U}_t:=\set{x\in\T\times\R:(t,x)\in \mathscr{U}},~t\in(0,T).
\]
As a consequence of \eqref{eq:bound_on_partialy2_f} the maps $X_t:\Omega_0\rightarrow \mathscr{U}_t$,
\[
X_t(y):=\begin{pmatrix}
y_1\\ty_2+f(t,y)
\end{pmatrix}, ~t\in(0,T)
\]
are $\cC^1$ diffeomorphisms with the property that also the joint maps $(0,T)\times\Omega_0\rightarrow \T\times\R$, $(t,y)\mapsto X_t(y)$ and $\mathscr{U}\rightarrow\T\times\R$, $(t,x)\mapsto X_t^{-1}(x)$ are of class $\cC^1$.

In view of \eqref{eq:ansatz} we thus can indeed define the density $\rho:[0,T)\times\T\times\R\rightarrow\R$ by setting
\begin{align*}
    \rho(t,x):=\begin{cases}
    1,&x_2\geq 2t+f(t,x_1,2),\\
    \frac{1}{2}(X_t^{-1}(x))_2,&x\in\mathscr{U}_t,\\
    -1,&x_2\leq -2t+f(t,x_1,-2)
    \end{cases}
\end{align*}
for $t\in(0,T)$ and $\rho(0,x):=\rho_0(x)$. Here $(X_t^{-1}(x))_2$ denotes the second component of $X_t^{-1}(x)$. Observe that $\rho$ is continuous except at points $(0,x_1,\gamma_0(x_1))$, $x_1\in\T$, and piecewise $\cC^1$ with the exceptional set being $\partial\mathscr{U}\subset [0,T)\times\T\times\R$. Moreover, as long as $t$ is positive $\rho(t,\cdot)$ is Lipschitz continuous and there exists a constant $C_0>0$ depending on the initial data such that
\begin{equation}\label{eq:bound_on_nabla_rho}
    \abs{\nabla\rho(t,x)}\leq \frac{C_0}{t}\mathbbm{1}_{\mathscr{U}_t}(x)
\end{equation}
for all $(t,x)\notin \partial \mathscr{U}$. 

Moreover, standard elliptic estimates show that $v$ defined through \eqref{eq:biot_savart}, \eqref{eq:initial_velocity_away_from_interface}, \eqref{eq:initial_velocity_at_interface} is the unique $L^2$ solution of the last two equations of \eqref{eq:subequation}. 

The stated log-Lipschitz continuity of $v(t,\cdot)$, $t>0$ is a consequence of the Biot-Savart operator acting on a compactly supported $L^\infty$-vorticity, cf. \cite{Marchioro_Pulvirenti}. In addition it is easy to see that also $v:[0,T)\times\T\times\R\rightarrow\R^2$ is continuous except at the one-dimensional set $\set{(0,x_1,\gamma_0(x_1)):x_1\in\T}$.

Hence we have shown properties \ref{eq:property_i_mainthm}, \ref{eq:property_ii_mainthm} of Theorem \ref{thm:main}. Moreover, observe that property \ref{eq:property_iii_mainthm} holds by construction with $\gamma_t$ given by
\[
\gamma_t(x_1,h):=t2h+f(t,x_1,2h),~x_1\in\T,~h\in[-1,1].
\]
It thus remains to show that the first equation of \eqref{eq:subequation} and the entropy balances \eqref{eq:entropy_with_equality} are satisfied. 

The regularity of $\rho$ implies
\[
\int_{\T\times\R}\rho(t,\cdot) v(t,\cdot)\cdot\nabla \varphi\:dx=-\int_{\T\times\R}v(t,\cdot)\cdot\nabla \rho(t,\cdot)\varphi\:dx
\]
for all $t\in(0,T)$, $\varphi\in\cC^\infty(\T\times\R)$. It follows that $(\rho,v)$ is a solution in the sense of Definition \ref{def:solution_ipmsub} if and only if the transport form
\begin{equation}\label{eq:transport_form}
\partial_t\rho+v\cdot\nabla \rho +2\rho\partial_{x_2}\rho=0
\end{equation}
of the equation is satisfied pointwise in $(0,T)\times\T\times\R\setminus\partial\mathscr{U}$. 

At points $(t,x)\notin \overline{\mathscr{U}}$ equation \eqref{eq:transport_form} trivially holds true. Inside $\mathscr{U}$ one can check that the computations in Section \ref{sec:ansatz_1} are possible showing that \eqref{eq:transport_form} is equivalent to \eqref{eq:equation_for_f}. Note that in Section \ref{sec:ansatz_1} we have formally assumed that $X_t$ are global diffeomorphisms $\T\times\R\rightarrow\T\times\R$, but as the reader can easily see, it is enough to have transformations from $\Omega_0$ to the corresponding $\mathscr{U}_t$.  

Observing that also the computations in Section \ref{sec:transformation_v} are legal in our scenario one sees that \eqref{eq:transport_form} on $\mathscr{U}$ is indeed equivalent to \eqref{eq:equation_for_f_after_conclusion}. 

Finally, let $\eta:\R\rightarrow\R$ be an arbitrary Lipschitz continuous function and define $Q:\R\rightarrow\R$,
\[
Q(u):=\int_0^u2\eta'(s)s\:ds,
\]
which is also Lipschitz continuous when restricted to any compact interval of $\R$. Consequently we have enough regularity to deduce \eqref{eq:entropy_with_equality} by multiplying \eqref{eq:transport_form} with $\eta'(\rho(t,x))$ and applying chain rule.
This finishes the proof of Theorem \ref{thm:main}.
\end{proof}

\appendix

\section{The abstract Cauchy-Kovalevskaya Theorem}\label{sec:CK_proof}

\begin{proof}[Proof of Theorem \ref{thm:abstract_CK}] As indicated in Section \ref{sec:solving} the proof of Theorem \ref{thm:abstract_CK} is a slight modification of the original proof in \cite{nishida}.

Let $a_0>0$ and set $a_{k+1}:=a_k(1-(k+2)^{-2})$, $k=0,1,\ldots$. Then
\[
a:=\lim_{k\rightarrow\infty}a_k>0.
\]
For $\rho\in(0,\rho_0)$ and $k=0,1,\ldots$ we define the intervals
\[
I_{k,\rho}:=\set{t\in[0,T):c(t)<a_k(\rho_0-\rho)}.
\]
We also define for a function $u$ with $u:I_{k,\rho}\rightarrow B_\rho$ continuous for any $\rho\in(0,\rho_0)$ the norm
\[
M_k[u]:=\sup\set{\norm{u(t)}_\rho\left(\frac{a_k(\rho_0-\rho)}{c(t)}-1\right):\rho\in(0,\rho_0),~t\in I_{k,\rho}}.
\]
Note that for $c(t)=t$ one recovers Nishida's set up. Now one recursively constructs the sequence
\[
u_0(t):=0, \quad u_{k+1}(t):=\begin{cases}
\frac{1}{a(t)}\int_0^tF(u_k(s),s)\:ds,&t\in(0,T),\\
0,&t=0.
\end{cases}
\] 

We claim that for $a_0$ chosen sufficiently small the recursion is well-defined, that each $u_k:I_{k,\rho}\rightarrow B_\rho$ is continuous with $\norm{u_k(t)}_\rho<R/2$ for $t\in I_{k,\rho}$, $\rho\in(0,\rho_0)$, and that 
\begin{equation}\label{eq:condition_on_lambdaks}
\lambda_{k-1}:=M_k[u_{k}-u_{k-1}]\leq (4Ka_0)^{k-1}a_0,
\end{equation}
where $K>0$ is the constant appearing in \eqref{eq:relation_between_functions_abc}.

We first of all note that $u_1(t)$ exists and satisfies the stated continuity condition due to assumptions \ref{cond:F_well_defined}, \ref{cond:F_0_well_defined}. Moreover, for $t\in I_{0,\rho}$ there holds $\norm{u_1(t)}_\rho<a_0$. Thus we pick $a_0$ at least smaller than $R/2$. One also easily checks that $\lambda_0\leq a_0$.

From now on we proceed by induction. Assume that the recursion with the above stated properties is possible up to some $k\geq 1$. Then it is clear that also $u_{k+1}:I_{k+1,\rho}\rightarrow B_\rho$ is well-defined, as well as continuous on the open interval $I_{k+1,\rho}\setminus \{0\}$ for any $\rho\in(0,\rho_0)$.

If we assume for now that \eqref{eq:condition_on_lambdaks} also holds true for $\lambda_k$, then for $t\in I_{k+1,\rho}$ we obtain in analogy to \cite{nishida} the estimate
\begin{align*}
\norm{u_{k+1}(t)}_\rho&\leq\sum_{j=0}^k\lambda_j\left(\frac{a_j(\rho_0-\rho)}{c(t)}-1\right)^{-1}\leq  \sum_{j=0}^k\lambda_j\left(\frac{a_j}{a_{j+1}}-1\right)^{-1}\\
&\leq a_0\sum_{j=0}^k(4Ka_0)^j (j+2)^2<R/2
\end{align*}
by choice of $a_0$ independent of $k$. Moreover, the first inequality in the above line of estimates applied at times $t>0$ with $c(t)<\frac{a}{2}(\rho_0-\rho)$ also gives
\begin{equation*}
\norm{u_{k+1}(t)}_{\rho}\leq c(t)\sum_{j=0}^k\frac{\lambda_j}{a_j(\rho_0-\rho)-c(t)}\leq \frac{2a_0c(t)}{a(\rho_0-\rho)}\sum_{j=0}^k(4Ka_0)^j,
\end{equation*}
which shows that $u_{k+1}$ is also continuous with respect to $\norm{\cdot}_\rho$ at $t=0$. 

To finish the induction 
it thus remains to show \eqref{eq:condition_on_lambdaks} for $\lambda_k$. The clever move is to use the contraction property with a different Banach space at each time $\tau$ inside the integral. Namely exactly as in \cite[p. 631]{nishida}, the contraction property of $F$ (Theorem~\ref{thm:abstract_CK} part \ref{cond:F_lipschitz}) with  $\rho(\tau):=\frac{1}{2}\left(\rho_0-\frac{c(\tau)}{a_k}+\rho\right)$ and the definition of $\lambda_{k-1}$ lead to
\begin{align*}
\norm{u_{k+1}(t)-u_k(t)}_\rho\leq \frac{4\lambda_{k-1}a_k}{a(t)}\int_0^t\frac{b(\tau)c(\tau)}{(a_k(\rho_0-\rho)-c(\tau))^2}\:d\tau
\end{align*}
for $t\in I_{k,\rho}$. At this point we use \eqref{eq:relation_between_functions_abc} and change of variables to obtain
\begin{align*}
\norm{u_{k+1}(t)-u_k(t)}_\rho\leq 4\lambda_{k-1}a_k Kc(t)\int_0^{c(t)}\frac{1}{(a_k(\rho_0-\rho)-\xi)^2}\:d\xi,
\end{align*}
from where one can conclude $\lambda_k\leq 4K\lambda_{k-1}a_0$ by following \cite{nishida} again.

Now Theorem \ref{thm:abstract_CK} follows as in \cite{nirenberg,nishida}. 
\end{proof}

\section{More on Otto's relaxation}\label{sec:otto_relaxation_appendix}

We here add some more details regarding the fifth step of Otto's relaxation \cite{otto1999} in the general non-flat case, which has only been sketched in Section \ref{sec:otto}.

Before doing that we will quickly convince ourselves that the setting in \cite{otto1999} is indeed equivalent to the formulation of IPM considered in our paper. Otto considers the equations
\begin{align}\begin{split}\label{ottos}
\partial_t s + u\cdot\nabla s =&0,\\
\nabla\cdot u =&0,\\
u=&-\nabla p +se_2,\end{split}
\end{align}
which correspond with \cite{otto1999}-(1.1)-(1.2) and the first equation on page 875 of \cite{otto1999} with $\lambda=1$. The parameter $\lambda$ in \cite{otto1999} is the quotient of the mobilities. In our case, we have taken both mobilities equal to one and then $\lambda=1$.  More importantly, in \cite{otto1999}, \begin{align}\label{svalues}s=\{0,1\},\end{align} however \begin{align}\label{rhovalues}\rho=\{-1,1\}\end{align} in our case.  

Let us see how we can go from \eqref{eq:ipm},\eqref{rhovalues} to \eqref{ottos}-\eqref{svalues}. Firstly we define 
\begin{align*}
    \bar{s}=\frac{1-\rho}{2}, && \rho=1-2\bar{s}
\end{align*}
thus
\begin{align*}
    \pa_t\bar{s}+v\cdot\nabla \bar{s}=&0\\
    \nabla \cdot v=&0\\
    v=&-\nabla (p+x_2)+2\bar{s}e_2,
\end{align*}
with  $$ \bar{s}=\{0,1\}.$$

We define $\bar{u}=v/2$ and $\bar{\Pi}=(p+x_2)/2$ which yields
\begin{align*}
    \pa_t\bar{s}+2\bar{u}\cdot\nabla \bar{s}=&0\\
    \nabla \cdot \bar{u}=&0\\
    \bar{u}=&-\nabla \bar{\Pi}+\bar{s}e_2,
\end{align*}

Finally we take $s(x,t)=\bar{s}(x,t/2)$, $u(x,t/2)=\bar{u}(x,t/2)$ and $\Pi(x,t)=\bar{\Pi}(x,t/2)$ thus
\begin{align*}
    \pa_t s+u\cdot\nabla s=&0\\
    \nabla \cdot u=&0\\
    u=&-\nabla \Pi+s e_2,
\end{align*}
with $s=\{0,1\}$, which agrees with \eqref{ottos}-\eqref{svalues} (up to a relabeling of the pressure). Therefore, if we show that \eqref{ottos}-\eqref{svalues} relaxes to 
\begin{align}\begin{split}\label{srelajada}
    \pa_t s +u \cdot \nabla s +\pa_{x_2} s-2 s\pa_{x_2} s =&0,\\
    \nabla \cdot u =&0,\\
    u=&-\nabla \Pi +s e_2,\end{split}
\end{align}
with $s\in [0,1]$, by undoing the previous transformations, we  see that \eqref{eq:ipm},\eqref{rhovalues} relaxes
to 
\begin{align}\begin{split}\label{ipmrelajada}
    \pa_t \rho +v\cdot \nabla \rho+2\rho\pa_{x_2}\rho=&0,\\
    \nabla\cdot v =&0,\\
    v=&-\nabla p -\rho e_2,\end{split}
\end{align}
with $\rho\in [-1,1].$

Next, we begin our formal discussion with the outcome of the fourth step of Otto, after which there exists for each $h>0$ a sequence of ``coarse grained'' functions $\{\theta^{k}\}_{k=0}^{N(h)}$ that are characterized by the following JKO scheme (which we understand as a minimizing movements scheme with respect to the Wasserstein distance): 
   
   $\theta^{(k+1)}$ is the minimizer in $K$ of
   \begin{align}\label{eq:JKO_in_appendix}
   \frac{1}{2}\text{dist}^2(\theta^{(k)},\theta)+\frac{1}{2}\text{dist}^2(1-\theta^{(k)},1-\theta)
   -h\int \theta(x)x_2
   \end{align}
   where the set $K$ consists of measurable $\theta$ taking values in $[0,1]$ and such that $\int \theta=\int s(x,0)$, and 
   $\text{dist}^2(\theta_0,\theta_1)$ is the $L^2$-Wasserstein distance
   \begin{align*}
       \text{dist}^2(\theta_0,\theta_1)=\inf_{\Phi\in I(\theta_0,\theta_1)}\int \theta_0(x)|\Phi(x)-x|^2dx
   \end{align*}
   with
   \begin{align*}
       I(\theta_0,\theta_1)=\{\Phi\, : \, \int \theta_1(y)\zeta(y)dy=\int \theta_0(x)\zeta(\Phi(x))dx \quad \forall \zeta \in \cC^0_0\}. 
   \end{align*}
   In the definition of $I(\theta_0,\theta_1)$, we have been deliberately imprecise and defer the reader to \cite{otto1999} for the proper definition. Even more,  
   in order to make the exposition more clear in the following we will assume that the minimizer exists, that it is smooth and that it satisfies pointwise the corresponding  Monge-Ampere equation. I.e.,
   \begin{align*}
       I(\theta_0,\theta_1)=\{\Phi \text{ diffeomorphism} : (\theta_1\circ \Phi)(x) J_{\Phi}(x)=\theta_0(x)\}. 
   \end{align*}  
   Here $J_\Phi$ denotes the Jacobian determinant $\det D\Phi$.

As explained in Section \ref{sec:otto} our goal is to show, on a formal level, that the limit as $h\rightarrow 0$, we will assume that it exists in the first place, of the functions 
\[
\theta_h(x,t):=\theta^{(k)}(x),\quad t\in[kh,(k+1)h))
\]
is characterized by system \eqref{srelajada}.

We begin with the Euler-Lagrange equation of \eqref{eq:JKO_in_appendix}. For a given $\theta_0\in K$ let $\theta_1$ be the minimizer in $K$ of 
\begin{align*}
    F[\theta]\equiv\frac{1}{2}\text{dist}^2(\theta_0,\theta)+\frac{1}{2}\text{dist}^2(1-\theta_0,1-\theta)
   -h\int \theta(x)x_2.
\end{align*}
Then we have that 
$$D_\theta F[\theta_1]\psi= \left.\frac{d}{d\tau}F[\theta_1+\tau\psi]\right|_{\tau=0}=0,$$
where we simply assume that $\theta_1+\tau\psi\in K$, i.e. we in particular consider $\psi$ with $\int\psi =0$. In order to compute $D_\theta F[\theta_1]\psi$ we first look at $D_\theta \text{ dist}^2(\theta_0,\theta_1)\psi$. Let $\Phi^{\tau}_0\in I(\theta_0,\theta_1+\tau\psi)$ be such that
\begin{align*}
     \text{dist}^2 (\theta_0,\theta_1+\tau\psi)=\inf_{\Phi\in I(\theta_0,\theta_1+\tau\psi)}
    \int \theta_0(x)|\Phi(x)-x|^2 dx=\int \theta_0(x)|\Phi^{\tau}_0(x)-x|^2dx.
\end{align*}
We define $$w\circ \Phi^0_0=\left.\frac{d \Phi^\tau_0}{d\tau}\right|_{\tau=0},$$
thus
\begin{align}\label{tres}
    \frac{1}{2}D_\theta \text{ dist}^2(\theta_0,\theta_1)\psi=\int \theta_0(x)(\Phi^0_0(x)-x)\cdot (w\circ \Phi_0^0)(x) dx.
\end{align}

We next compute for which $w$ we have $\Phi^\tau_0\in I(\theta_0,\theta_1+\tau \psi)$.
There holds
\begin{align}\label{bI}
    J_{\Phi^\tau_0}(x)((\theta_1+\tau\psi)\circ \Phi^\tau_0)(x)=\theta_0(x).
\end{align}
Taking a $\tau-$derivative in \eqref{bI} and evaluating at $\tau=0$ yields,
\begin{align*}
   J_{\Phi_0^0} \text{div }w\circ \Phi^0_0\theta_1\circ \Phi^0_0+J_{\Phi^0_0}w\circ\Phi^0_0\cdot \nabla \theta_1\circ\Phi_0^0+J_{\Phi^0_0}\psi\circ \Phi^0_0=0 
\end{align*}
which reduces to
\begin{align}\label{cinco}
    \text{div }(w\theta_1)+\psi=0.
\end{align}

In addition,  $\Phi^0_0$ minimizes $$\int \theta_0(x)|\Phi(x)-x|^2 dx$$ in $I(\theta_0,\theta_1)$. So for every family of flows 
$(\Phi^0_\delta)\in I(\theta_0,\theta_1)$ we have that
\begin{align*}
    \left.\frac{d}{d\delta} \int \theta_0(x)|\Phi^0_\delta(x)-x|^2 dx\right|_{\delta=0} =0.
\end{align*}
That is
\begin{align*}
    \int \theta_0(x)(\Phi^0_0(x)-x)\cdot (\bar{w}\circ \Phi^0_0)(x)dx=0,
\end{align*}
where if $\Phi_\delta$ is the flow of a vector field $\bar{w}$, 
\begin{align}\label{uno}
    \bar{w}\circ \Phi^0_0&=\left.\frac{d \Phi^0_\delta}{d\delta}\right|_{\delta=0},\\
\text{div} (\theta_1 \bar{w})&=0\label{dos}.
\end{align}

The condition \eqref{dos}, equivalent to  $\Phi_\delta \in I(\theta_0,\theta_1)$, is deduced by 
differentiating 
\begin{align}\label{seis}
    J_{\Phi_\delta^0}(x)(\theta_1\circ \Phi_\delta^0)(x)=\theta_0(x)
\end{align}
with respect to $\delta$. 

Therefore, it holds that 
\begin{align*}
    0=\int \theta_0(x)(\Phi^0_0(x)-x)\cdot (\bar{w}\circ \Phi^0_0)(x)dx=\int \theta_1(x)\bar{w}(x)\cdot (x-\left(\Phi^0_0\right)^{-1}(x)) dx,
\end{align*}
where in the last equality we have used the definition of $I(\theta_0,\theta_1)$.
Since $\bar w$ is an arbitrary vector field,
Hodge decomposition implies that
\begin{align}\label{cuatro}
    x-\left(\Phi^0_0\right)^{-1}(x)=\nabla a(x)
\end{align}
for some function $a$. In order to avoid technicalities we here have implicitly assumed that $\theta_1$ does not vanish.

From \eqref{tres}, \eqref{cinco} and \eqref{cuatro} we see that
\begin{align*}
    \frac {1}{2} D_\theta \text{ dist}^2 (\theta_0,\theta_1)\psi =&\int \theta_1(x)w(x)\cdot\nabla a(x) dx
    =-\int \nabla \cdot (\theta_1 w)(x)a(x)dx\\
    =& \int \psi(x)a(x).
\end{align*}
We have obtained that
\begin{align*}
    \frac {1}{2} D_\theta \text{ dist}^2 (\theta_0,\theta_1)\psi =& \int \psi(x)a(x)\\
    \Phi_0^0(x)=& x+(\nabla a \circ \Phi^0_0)(x).
\end{align*}
Similar computations yield
\begin{align*}
    \frac {1}{2} D_\theta \text{ dist}^2 (1-\theta_0,1-\theta_1)\psi =& -\int \psi(x)\bar{a}(x)\\
    \bar{\Phi}_0^0(x)=& x+(\nabla \bar{a} \circ \bar{\Phi}^0_0)(x),
\end{align*}
and putting everything together we arrive at 
\begin{align}\label{siete}
    &D_\theta F[\theta_1]\psi=\int (a(x)-\bar{a}(x)-hx_2)\psi(x)dx=0
\end{align}
for all $\psi$ with $\int\psi =0$.
Moreover, since $\Phi^0_0\in I(\theta_0,\theta_1)$ and $\bar{\Phi}^0_0\in I(1-\theta_0,1-\theta_1)$, 
\begin{align}
    \theta_1(x)&=J_{(\Phi_0^0)^{-1}}(x)\theta_0(x-\nabla a(x)),\label{ocho}\\
    (1-\theta_1)(x)&=J_{(\bar{\Phi}_0^0)^{-1}}(x)(1-\theta_0)(x-\nabla \bar{a}(x)).\label{nueve}
\end{align}

Note that so far we have omitted the $h$-dependence of the functions $a,\bar a, \Phi_0^0,\bar{\Phi}_0^0$ in our notation. We continue doing so when introducing $p=\frac{a}{h},\bar{p}=\frac{\bar{a}}{h}$ which, up to a constant satisfy
$$ p-\bar{p}=x_2$$
by \eqref{siete}. Note that the constant is irrelevant since only derivatives of $p$, $\bar p$ will play a role. To obtain a formal limit as $h\rightarrow 0$, we will assume in the following that the $h$-dependent functions $p$, $\bar{p}$ have a well-defined $\cC^2$ limit, which will again be denoted by $p$, $\bar p$.

Now we take said limit. On one hand we have from \eqref{ocho} that
\begin{align*}
    &\frac{\theta_1(x)-\theta_0(x)}{h}= \frac{J_{(\Phi^0_0)^{-1}}(x)\theta_0(x-h\nabla p(x))-\theta_0(x)}{h}\\&\hspace{35pt}=\frac{(J_{(\Phi^0_0)^{-1}}(x)-1)\theta_0(x-h\nabla p(x))+\theta_0(x-h\nabla p(x))-\theta_0(x)}{h}.
\end{align*}
Recall that $\Phi_0^0$ is linked to $p$ via \eqref{cuatro} and thus since $\Phi_0^0(x)\rightarrow x$, it holds that 
\begin{align*}
    J_{(\Phi_0^0)^{-1}}(x)-1=-h\Delta p(x),
\end{align*}
at first order in $h$. Thus when letting $h\rightarrow 0$ in the difference quotient we arrive at 
\begin{align}\label{diez}
    \pa_t \theta = -\Delta p \,\theta-\nabla p \cdot \nabla \theta.
\end{align}

On the other hand we have from \eqref{nueve} that
\begin{align*}
\frac{\theta_1(x)-\theta_0(x)}{h}=& \frac{1-J_{(\bar{\Phi}^0_0)^{-1}}(x)+J_{(\bar{\Phi}^0_0)^{-1}}(x)\theta_0(x-h\nabla \bar{p}(x))-\theta_0(x)}{h}.
\end{align*}
Passing to the limit yields
\begin{align}\label{once}
  \pa_t\theta=  \Delta \bar{p}-\Delta \bar{p}\theta -\nabla \theta\cdot \nabla \bar{p}.
\end{align}

In order for \eqref{diez} and \eqref{once} to agree, there holds
\begin{align*}
    \Delta \bar{p} -\nabla\cdot (\nabla \bar{p} \theta)=-\nabla\cdot (\nabla p \theta)
\end{align*}
and since $p=\bar{p}+x_2$ we have that
\begin{align}\label{veinticuatro}
    \Delta \bar{p}=-\nabla\cdot (\nabla x_2 \theta)=-\pa_{x_2}\theta.
\end{align}

Therefore, from \eqref{once} and \eqref{veinticuatro},
\begin{align*}
    \pa_t\theta=& -\pa_{x_2}\theta -\nabla \cdot (\nabla \bar{p}\theta)\\
    =& -\pa_{x_2}\theta -\nabla\cdot (\left(\nabla\bar{p}+\theta e_2\right)\theta)+\nabla \cdot (\theta^2e_2)
\end{align*}
To finish we define $u=\nabla \bar{p}+\theta e_2$, which clearly satisfies $\nabla \cdot u=0$, to get
\begin{align*}
    \pa_t \theta +u\cdot \nabla  \theta +\pa_{x_2}\theta -2\theta \pa_{x_2}\theta=&0,\\
    u=&\nabla \bar{p}+\theta e_2,\\
    \nabla \cdot u =&0,
\end{align*}
which agrees with \eqref{srelajada}.

\section{Rigorous energy dissipation}\label{sec:rigorous_dissipation}

In Section \ref{sec:transfer_to_subsolutions} equation \eqref{eq:energy_dissipation} we have formally computed the decay rate of the total potential energy. For completeness we give sufficient conditions when this computation is justified. Also for completeness we show that the subsolution given by Theorem \ref{thm:main} indeed satisfies the sufficient conditions. 
\begin{lemma}\label{lem:rigorous_dissipation} Let $\rho_0\in L^1_{loc}(\T\times\R)$ be some initial data and suppose that the pair of functions $(\rho,m)\in L^1_{loc}((0,T)\times\T\times\R;\R\times\R^2)$ satisfies
\begin{align*}
    \partial_t\rho+\divv m=0,\quad \rho(0,\cdot)=\rho_0
\end{align*}
on $(0,T)\times\T\times\R$ in the sense of distributions. If there exists $\alpha>0$ such that 
\begin{gather*}
m_2,~(\rho-\rho_0)x_2,~(\rho-\rho_0)\abs{x_2}^{1+\alpha}\in \cC^0([0,T);L^1(\T\times\R)),\\
m_2\abs{x_2}^\alpha\in L^1((0,T)\times\T\times\R),
\end{gather*}
then the relative potential energy defined in \eqref{eq:relative_potential_energy} belongs to $\cC^1([0,T))$ and there holds
\[
\frac{d}{dt}E_{rel}(t)=\int_{\T\times\R}m_2(t,x)\:dx.
\]
\end{lemma}
\begin{proof}
Let $R>0$ and $\varphi_R:\R\rightarrow[0,1]$ be a cutoff function with $\varphi_R(x_2)=1$ for $\abs{x_2}\leq R$, $\varphi_R(x_2)=0$ for $\abs{x_2}\geq 2R$ and $\abs{\varphi_R'(x_2)}\leq 2R^{-1}$, $x_2\in\R$.

We abbreviate $E(t)=E_{rel}(t)$ and define
\[
E_R(t):=\int_{\T\times\R}(\rho(t,x)-\rho_0(x))x_2\varphi_R(x_2)\:dx.
\]
Note that $E(t)$, $E_R(t)$ are well-defined at every time $t\in[0,T)$ and there holds
\begin{align*}
    \abs{E(t)-E_R(t)}\leq \norm{(\rho(t,\cdot)-\rho_0)\abs{x_2}^{1+\alpha}}_{L^1(\T\times\R)}\frac{1}{R^\alpha}=O(R^{-\alpha})
\end{align*}
uniformly in time as $R\rightarrow\infty$.
Thus 
\begin{align*}
    h^{-1}(E(t+h)-E(t))&=h^{-1}(E_R(t+h)-E_R(t))+h^{-1}O(R^{-\alpha}).
\end{align*}
Moreover, the assumed continuity conditions and approximation of the indicator function of $[t,t+h]$ imply
\begin{align*}
    E_R(t+h&)-E_R(t)=\int_t^{t+h}\int_{\T\times\R}m\cdot \nabla (x_2\varphi_R(x_2))\:dx\:ds\\
    &=\int_t^{t+h}\int_{\T\times\R}m_2\:dx\:ds+\int_t^{t+h}\int_{\T\times\R}m_2(\varphi_R(x_2)-1+x_2\varphi_R'(x_2))\:dx\:ds.
\end{align*}
Now the latter term can be bounded by $5R^{-\alpha}\norm{m_2\abs{x_2}^\alpha}_{L^1((0,T)\times\T\times\R)}$ implying that 
\begin{align*}
    h^{-1}(E(t+h)-E(t))=h^{-1}\int_t^{t+h}\int_{\T\times\R}m_2\:dx\:ds+h^{-1}O(R^{-\alpha}).
\end{align*}
The statement follows.
\end{proof}

Let $(\rho,v)$ be the solution constructed in Theorem \ref{thm:main} and set
\[
m=\rho v-(1-\rho^2)e_2.
\]
\begin{lemma} In addition to the properties stated in Theorem \ref{thm:main} the velocity field $v$ satisfies
    $$
    \abs{v(t,x)}\leq Ce^{-\abs{x_2}}
    $$ whenever $\abs{x_2}\geq R$ for constants $C,R>0$ independent of $t$. The pair $(\rho,m)$ in particular satisfies the conditions of Lemma \ref{lem:rigorous_dissipation}.
\end{lemma}
\begin{proof}
Regarding the second component one easily sees that 
\begin{align*}
\abs{v_2(t,x)}\leq \abs{\mathscr{U}_t}\norm{\partial_{x_1}\rho(t,\cdot)}_{L^\infty(\T\times\R)}\norm{K_2(x-\cdot)}_{L^\infty(\mathscr{U}_t)}
\end{align*}
which can be bounded by $Ce^{-\abs{x_2}}$ for $\abs{x_2}\geq R$ with constants $C,R>0$ independent of time.

Regarding the first component we can not exploit the decay of the kernel, since $K_1(z)\rightarrow \mp 1$ as $z_2\rightarrow\pm \infty$. Still by substracting vanishing horizontal averages we deduce
\begin{align*}
    \abs{v_1(t,x)}&=\abs{\int_{\T\times\R}\partial_{x_1}\rho(t,y)(K_1(x-y)-K_1(x-(0,y_2))\:dy}\\
    &\leq \abs{\mathscr{U}_t}\norm{\partial_{x_1}\rho(t,\cdot)}_{L^\infty(\T\times\R)}\norm{\partial_{z_1}K_1(x-\cdot)}_{L^\infty(\mathscr{U}^*_t)}\pi,
\end{align*}
where $\mathscr{U}_t^*$ is the set of points obtained by taking all segments between $y\in \mathscr{U}_t$ and $(0,y_2)$. It is only important that those sets are bounded uniformly in time, which allows to argue as above for $v_2$, since $\partial_{z_1}K_1$ now has the required decay.
\end{proof}

\bibliographystyle{abbrv}
\bibliography{references_angel}

\vspace{5pt}
\noindent\textbf{Acknowledgements}. A.C., D.F. and B.G. acknowledge financial support from  the Severo Ochoa Programme for Centres of Excellence Grant CEX2019-000904-S  funded by MCIN/AEI/10.13039/501100011033. A.C. from Grant PID2020-114703GB-I00 funded by MCIN/AEI/10.13039/501100011033 and from a 2023 Leonardo Grant for Researchers and Cultural Creators, BBVA Foundation. The BBVA Foundation accepts no responsibility for the opinions, statements, and contents included in the project and/or the results thereof, which are entirely the responsibility of the authors. D.F. and B.G. from grant PI2021-124-195NB-C32 funded by MCIN/AEI/10.13039/501100011033. D.F. was also partially supported by CAM through the Line of excellence for University Teaching Staff between CM and UAM.  D.F. and B.G. were also partially supported by the ERC Advanced Grant 834728. B.G. has been supported by Mar\'ia Zambrano Grant CA6/RSUE/2022-00097 and is partially also funded by the Deutsche Forschungsgemeinschaft (DFG, German Research Foundation) under Germany's Excellence Strategy EXC 2044-390685587, Mathematics Münster: Dynamics-Geometry\\-Structure. Finally A.C. and D.F. acknowledge financial support from Grants RED2022-134784-T and RED2018-102650-T funded by \\ MCIN/AEI/10.13039/501100011033.

\vspace{10pt}
\noindent \'Angel Castro\\
Instituto de Ciencias Matem\'aticas, CSIC-UAM-UC3M-UCM, 28049 Madrid, Spain\\
angel\_castro@icmat.es

\vspace{10pt}
\noindent Daniel Faraco\\
Departamento de Matem\'aticas, Universidad Aut\'onoma de Madrid, Instituto de Ciencias
Matem\'aticas, CSIC-UAM-UC3M-UCM, 28049 Madrid, Spain\\
daniel.faraco@uam.es

\vspace{10pt}

\noindent Bj\"orn Gebhard\\
Mathematics M\"unster, University of M\"unster, 48149 Münster, Germany\\
bjoern.gebhard@uni-muenster.de

\end{document}